\documentclass[a4paper, 12pt]{amsart}   
\usepackage{mathptmx, amssymb,amscd,latexsym, eulervm}   
\usepackage{amsmath}
\usepackage{amsthm}
\usepackage{mathdots}
\usepackage[pagebackref,colorlinks=true,linkcolor=blue,urlcolor=blue]{hyperref}
\usepackage{color}
\usepackage[onehalfspacing]{setspace}
\usepackage{comment}
\usepackage{tabularx}
\usepackage{amsfonts}
\usepackage{paralist}
\usepackage{aliascnt}
\usepackage[initials, lite]{amsrefs}
\usepackage{amscd}
\usepackage{blkarray}
\usepackage{mathbbol}
\usepackage{setspace}
\usepackage[inner=2.5cm,outer=2.5cm, bottom=3.2cm]{geometry}
\usepackage{tikz, tikz-cd}
\usepackage{calligra,mathrsfs}

\usepackage{tikz}
\usetikzlibrary{matrix}
\usetikzlibrary{arrows,calc}
\allowdisplaybreaks

\BibSpec{collection.article}{%
	+{}  {\PrintAuthors}                {author}
	+{,} { \textit}                     {title}
	+{.} { }                            {part}
	+{:} { \textit}                     {subtitle}
	+{,} { \PrintContributions}         {contribution}
	+{,} { \PrintConference}            {conference}
	+{}  {\PrintBook}                   {book}
	+{,} { }                            {booktitle}
	+{,} { }                            {series}
	+{, vol.} { }                            {volume}
	+{,} { }                            {publisher}
	+{,} { \PrintDateB}                 {date}
	+{,} { pp.~}                        {pages}
	+{,} { }                            {status}
	+{,} { \PrintDOI}                   {doi}
	+{,} { available at \eprint}        {eprint}
	+{}  { \parenthesize}               {language}
	+{}  { \PrintTranslation}           {translation}
	+{;} { \PrintReprint}               {reprint}
	+{.} { }                            {note}
	+{.} {}                             {transition}
	+{}  {\SentenceSpace \PrintReviews} {review}
}

\newcommand{\sat}{{\normalfont\text{sat}}}

\newcommand{\pd}{\normalfont\text{pd}}

\newtheorem{theorem}{Theorem}[section]

\newtheorem{headthm}{Theorem}

\newaliascnt{headcor}{headthm}

\aliascntresetthe{headcor}

\newaliascnt{headconj}{headthm}

\aliascntresetthe{headconj}

\newaliascnt{corollary}{theorem}

\aliascntresetthe{corollary}

\newaliascnt{claim}{theorem}

\aliascntresetthe{claim}

\newaliascnt{lemma}{theorem}
\newtheorem{lemma}[lemma]{Lemma}
\aliascntresetthe{lemma}

\newaliascnt{conjecture}{theorem}

\aliascntresetthe{conjecture}

\newaliascnt{proposition}{theorem}

\aliascntresetthe{proposition}

\theoremstyle{definition}
\newaliascnt{definition}{theorem}
\newtheorem{definition}[definition]{Definition}
\aliascntresetthe{definition}

\newaliascnt{notation}{theorem}

\aliascntresetthe{notation}

\newaliascnt{example}{theorem}
\newtheorem{example}[example]{Example}
\aliascntresetthe{example}

\newaliascnt{examples}{theorem}

\aliascntresetthe{examples}

\newaliascnt{remark}{theorem}
\newtheorem{remark}[remark]{Remark}
\aliascntresetthe{remark}

\newaliascnt{question}{theorem}
\newtheorem{question}[question]{Question}
\aliascntresetthe{question}

\newaliascnt{questions}{theorem}

\aliascntresetthe{questions}

\newaliascnt{problem}{theorem}

\aliascntresetthe{problem}

\newaliascnt{construction}{theorem}

\aliascntresetthe{construction}

\newaliascnt{setup}{theorem}

\aliascntresetthe{setup}

\newaliascnt{algorithm}{theorem}

\aliascntresetthe{algorithm}

\newaliascnt{observation}{theorem}

\aliascntresetthe{observation}

\newaliascnt{defprop}{theorem}

\aliascntresetthe{defprop}

\DeclareFontFamily{OT1}{pzc}{}
\DeclareFontShape{OT1}{pzc}{m}{it}{<-> s * [1.100] pzcmi7t}{}
\DeclareMathAlphabet{\mathchanc}{OT1}{pzc}{m}{it}

\def\equationautorefname~#1\null{(#1)\null}
\def\sectionautorefname~#1\null{Section #1\null}
\def\subsectionautorefname~#1\null{\S #1\null}

\newcommand{\vol}{\text{Vol}}

\newcommand{\spec}{\text{Spec}}

\newcommand{\ash}{\text{Assh}}

\newcommand{\con}{\text{Cone}}

\newcommand{\cal}{\mathcal}

\makeatletter
\@namedef{subjclassname@2020}{
  \textup{2020} Mathematics Subject Classification}

\begin{document}

\title{Generalized Hilbert-Kunz multiplicity for families of ideals}

\author{Sudipta Das and Stephen Landsittel}
\address{(Das) Dr Homi Bhabha Rd, TIFR, Navy Nagar, Colaba, Mumbai, Maharashtra 400005}
\email{sudiptad@math.tifr.res.in}
\address
{(Landsittel) Institute of Mathematics, Hebrew University, Givat Ram, Jerusalem 91904, Israel}
\email{stephen.landsittel@mail.huji.ac.il}
\begin{abstract}
    In this paper, we initiate a systematic study of the generalized Hilbert–Kunz multiplicity for families  of ideals in a Noetherian local ring $(R,\mathfrak{m})$ of positive characteristic, and introduce a new asymptotic invariant called the \emph{Amao-type multiplicity}. We establish that, for a $p$-family of ideals, the generalized Hilbert–Kunz multiplicity arises as the limit of Amao-type multiplicities.
\end{abstract}
%Several explicit computations are provided to illustrate these notions.

\maketitle

\section{Introduction}

%why is $\text{vol} not rendering as normal font in theorem equations$

%%%%In definition 7.8 do we want weak p family??? instead of weakly???

In this article we consider an arbitrary $d$ dimensional local ring $(R,\mathfrak{m})$ in positive characteristic. Let $I \subset R$ be an $\mathfrak{m}$-primary ideal and $q=p^{e}$ for some $e \in \mathbb{N}$, and $I^{[q]}= \left( x^{q} \mid x \in I \right)$, the $q$-th Frobenius power of $I$. We denote the $\mathfrak{m}$-adic completion of $R$ by $\hat{R}$ and $\ell_{R}(-)$ denotes the length as an $R$-module. 

 Several important numerical invariants have been extensively studied in rings of positive characteristic. 
Among them, the \emph{Hilbert--Kunz} function plays a central role, as it captures the asymptotic behavior 
of the length $\ell_R(R/I^{[q]})$ as a function of $q$ and provides a refined measure of the singularities of the ring. In his work, Kunz \cite{K76} introduced this as a way to measure how close the ring $R$ is to being regular. Later P. Monsky \cite{M83} showed that \\
\begin{equation*}
\label{Equation 1.1}
    e_{HK}(I)=\lim_{e \to \infty} \dfrac{\ell_{R}(R/I^{[q]})}{q^{d}} 
\end{equation*} exists for any $\mathfrak{m}$-primary ideal $I$, this positive real number is called the \textit{Hilbert-Kunz Multiplicity} of $I$. 
This definition can naturally be extended to finitely generated modules as follows. Let $F_{*}^{e}R$ denote an $R$
algebra obtained by the $e$th iterate of the Frobenius endomorphism, i.e., $F_{*}^{e}R$ is abstractly
isomorphic to $R$ as a ring and $r.s =r^{p^e}.s$ for any elements $r \in R$ and  $ s \in F_{*}^{e}R$. Then for a finitely generated $R$-module $M$, we study the function $ f_{gHK}^{M}(e)= \ell_{F_{*}^{e}R}
\left(H_{\mathfrak{m}}^0(M \otimes_{R} F_{*}^{e}R)\right)$ and the limit $$e_{gHK}(M)= \displaystyle \lim_{e \to \infty} \dfrac{f_{gHK}^{M}(e)}{p^{ed}}$$ provided it exists. We refer to this quantity as the \emph{generalized Hilbert--Kunz function}, and its limiting value, when it exists, as the \emph{generalized Hilbert--Kunz multiplicity} of \( M \).
Throughout this paper, we focus on the case where \( M = R/I \) for an arbitrary ideal $I$ and we call it $e_{gHK}(I)$. Dao and Smirnov prove that this limit  exists when $R$ is a complete intersection with an isolated singularity in \cite{DS}. Jeffries and Hernández \cite{HJ} show that $e_{gHK}(I)$ exists for any ideal $I$ satisfying $\mathfrak{m}^{cq} \cap (I^{[q]})^{\sat}
= \mathfrak{m}^{cq} \cap I^{[q]}$ provided the ring is either an analytically irreducible local
ring, or a graded domain. Vraciu has shown in \cite{V1} that if every ideal $I$ of
$R$ satisfies a notion of interest in tight closure theory called (LC) condition then  $e_{gHK} (I)$ exists, where (LC) says there exists a positive integer $c$
such that $\mathfrak{m}^{cq} (I^{[q]})^{\sat} \subset I^{[q]}$ for all $q$. 

In Section \ref{sec2}, we use the description of  Epstien and Yao in \cite{EY} to define the  generalized Hilbert-Kunz multiplicity of the ideal $I$. This multiplicity is an useful numerical invariant for determining homological properties. For example if $(R,\mathfrak{m})$ is a complete intersection, then $e_{gHK}(M) = 0$ if and only if $ \pd (M) < \dim R$ ( \cite{DS}). In \cite{BrennerCaminata} Brenner and Caminata showed that when $R$ is a two-dimensional normal standard graded ring over an algebraically closed field of prime characteristic, generalized Hilbert–Kunz multiplicity of $R/I$ can be computed entirely using numerical invariants coming from the geometry of the syzygy bundle associated with the generators of $I$. In particular, this multiplicity depends on the degrees of the generators of $I$, the degree of the corresponding ideal sheaf, and the data from the strong Harder–Narasimhan filtration of that syzygy bundle. This generalized multiplicity and related analogues results are also utilized by Brenner \cite{Brenner}  in his proof of the existence of
irrational Hilbert-Kunz multiplicities

Motivated by this various interesting aspect of generalized Hilbert-Kunz multiplicity, in this paper, we have initiated the study of generalized Hilbert-Kunz multiplcity for $p$-families $\{I_q\}$, and give conditions for when the sequence $\ell(H^0_m(R/I_q))/q^d$ has a limit, where a $p$-family of ideals is a sequence of ideals $ I_{\bullet}=\{I_{q}\}_{q=1}^{\infty}$ with $I_{q}^{[p]} \subseteq I_{pq}$ for all $q$. We call this limit (when it exists) the generalized Hilbert-Kunz multiplicity of the $p$-family $\{I_q\}$, denoted by $e_{gHK}(I_{\bullet})$. Clearly for an $\mathfrak{m}$-primary ideal $I\subset R$, with $\{I_q=I^{[q]}\}$ we recover the usual Hilbert-Kunz multiplicity of $I$. \\
 The second multiplicity we introduce is $ a_F(I,J) = \displaystyle \limsup_{q \to \infty} \dfrac{\ell\left(J^{[q]}/I^{[q]}\right)}{q^d}$ where $J$ and $I$ are two ideals in $R$ such that $J/I$ has finite length, and we call this Amao-type multiplicity (see Definition \ref{df-at}). This is the Frobenius analogue of the Amao multicity, introduced in the paper of Amao \cite{Amao}. Note that the Amao multiplicity can be calculated from the leading coeffiecient of a polynomial, while the Amao-type multiplicity exhibits more complex behavior. Additionally, we don't know whether such a limit always exists.
 
 Throughout this paper we assume linear growth type condition (see Definition \ref{p(c)}) for $p$-families of ideals which has been discussed in \cite{HJ}. Under these conditions, we present the main result of this paper. We demonstrate that the generalized Hilbert-Kunz multiplicity of a $p$-family $\{I_q\}$ is equal to the asymptotic limit of the Amao-type multiplicities $a_F(I_q, (I_q)^{\sat})$. A special case of this result for a $p$-family of $\mathfrak{m}$-primary ideals has been done in \cite{D1,DC}.
 \phantom{}\\\\
 
\begin{headthm}(\autoref{thm6-1}, \autoref{thm6-2}) \label{Main_Theorem}
    Let $(R,m)$ be a $d$-dimensional regular local ring of prime characteristic $p>0$. Let $I_{\bullet} = \{I_q\}$ be a $p$-family in $R$ for which there is $c>0$ such that $I_q\cap m^{cq} = I_q^{\sat}\cap m^{cq}$ for all $q$. Suppose that for $P\in\ash(R)$ such that $P\supset I_1$, we have $P\supset I_q$ for all $q$. Then
    \begin{enumerate}
        \item[$(i)$]For all $q'$, the limit
        \begin{equation*}
            a_F(I_{q'}, (I_{q'})^{\sat}) = \lim_{q\to\infty}\frac{\ell_R([(I_{q'})^{\sat}]^{[q]}/(I_{q'}^{[q]}))}{q^d}
        \end{equation*}exists.
        \item[$(ii)$] The limit
        \begin{equation*}
            \lim_{q'\to\infty}\frac{a_F(I_{q'}, (I_{q'})^{\sat})}{(q')^d}
        \end{equation*}exists.
        \item[$(iii)$] We have the following formula
        \begin{equation*}
            e_{_gHK}(I_{\bullet})
            = \lim_{q'\to\infty}\frac{a_F(I_{q'}, (I_{q'})^{\sat})}{(q')^d}.
        \end{equation*}
    \end{enumerate}

    Moreover, let $(R, \mathfrak{m})$ is a $d$-dimensional analytically unramified local ring and $I \subset R$ be an ideal. Suppose the associated $p$-family $\{I^{[q]}\}$ satisfies for some $c > 0$, $I^{[q]}\cap m^{cq} = (I^{[q]})^{\sat}\cap m^{cq}$ for all $q$. Then the generalized Hilbert–Kunz multiplicity of $I$ is given by \begin{equation*}
            e_{_gHK}(I)
            = \lim_{q'\to\infty}\frac{a_F(I^{[q']}, (I^{[q']})^{\sat})}{(q')^d}.
        \end{equation*}
\end{headthm}
\phantom{}\\
The next technical result lies at the core of this paper. It captures the essential mechanism driving our main result
and provides the critical bridge connecting the behavior of $p$-families
to the asymptotic formula established in the main theorem. It ensures the existence of the relevant limits and ultimately connects the generalized Hilbert–Kunz multiplicity with the Amao-type multiplicities.

\begin{headthm} (\autoref{generalmult}) \label{Main_technical_result}
Suppose that $R$ is an analytically unramified local ring of dimension $d$. Let that $I_{\bullet} = \{I_q\}\subset J_{\bullet} = \{J_q\}$ be $p$-families. Suppose further that, for each $q' = p^e$, there exists $p$-families $I_{q',\bullet} = \{I_{q',q}\}\subset J_{q',\bullet} = \{J_{q',q}\}$ such that $J_{q',1} = J_{q'}$, $J_{q',q}\subset J_{qq'}$, $I_{q',1} = I_{q'}$, and $I_{q',q}\subset I_{qq'}$ for all $q$. 
Suppose that there is a positive number $c$ such that
\begin{equation*}
    I_{q',q}\cap m_R^{cqq'} = J_{q',q}\cap m_R^{cqq'}
\end{equation*}for all $q,q'$. Suppose further that whenever $P\in\ash(R)$ such that $P\supset I_1$, we have $P\supset I_q$ for all $q$. Let
\begin{equation*}
    \cal{F}(I_{\bullet},J_{\bullet}):= \lim_{q\to\infty}\frac{\ell_R(J_q/I_q)}{q^d}.
\end{equation*}
Then the above limit exists. Moreover,
\begin{equation*}
    \cal{G}(J_{q',\bullet}, I_{q',\bullet}):=\lim_{q\to\infty}
    \frac{\ell_R(J_{q',q}/I_{q',q})}{q^d}
\end{equation*}exists for all $q'$ and the limit
\begin{equation*}
    \lim_{q'\to\infty}\frac{\cal{G}(J_{q',\bullet}, I_{q',\bullet})}{(q')^d}
\end{equation*}exists, and we have the following formula
\begin{equation}
    \cal{F}(I_{\bullet},J_{\bullet})
    =\lim_{q'\to\infty}\frac{\cal{G}(J_{q',\bullet}, I_{q',\bullet})}{(q')^d}.
\end{equation}
\end{headthm}

In this paper, we have also established an alternative approach to proving the limit existence of the asymptotic colength of any BBL weakly $p$-family of $\mathfrak{m}$-primary ideals (see Definition \ref{df-weakp}), a result previously shown in \cite{DC,Tucker}. The aforementioned technical theorem is also instrumental in establishing the Volume = Multiplicity formula for this class of families.

\begin{headthm} (\autoref{thm7-1}) \label{Minkowski}
    Let $(R,\mathfrak{m})$ be a $d$-dimensional Noetherian local ring of characteristic $p>0$. Let $I_{\bullet} = \{I_q\}$ be a BBL weakly $p$-family of $\mathfrak{m}$-primary ideals in $R$. Suppose that the dimension of the nilradical of $\widehat{R}$ is less than $d$. Then we have the following.
    
    \begin{enumerate}%fix italics in the numbers
    \item[$(1)$]
        $\lim_{q\to\infty}\frac{\ell_R(R/I_q)}{q^d}$
    exists.
    
    \item[$(2)$]
    $\lim_{q\to\infty}\frac{\ell_R(R/I_q)}{q^d} = \lim_{q\to\infty}\frac{e_{HK}(I_q)}{q^d}$.
    \end{enumerate}
\end{headthm}

\section{Preliminaries}\label{sec3}

In this section, we introduce the essential definitions and notations that will be required in the subsequent sections.
\subsection{Cones in $\mathbb{R}^d$}

We begin by introducing some notation from analysis and convex geometry and which will be used later in this Section.

For a Lebesgue measurable subset $E\subset \mathbb{R}^d$, we will denote its Lebesgue measure by $\vol_{\mathbb{R}^d}(E)$. If $A,B$ are two subsets of $\mathbb{R}^d$, then we define the \textit{Minkowski} sum of $A$ and $B$ by $\{a+b\mid a\in A\text{, }b\in B\}$, which is denoted by $A+B$. If $E$ is a subset of $\mathbb{R}^d$, we denote it's closed convex cone by $\con(E)$, which is the closure of the set of all linear combinations $\displaystyle\sum_{i=1}^n\lambda_ix_i$, $\lambda_i\in \mathbb{R}_{\geqslant 0}$, $x_i\in E$, $n\geqslant \mathbb{Z}_{\geqslant 1}$. For $b,a\in\mathbb{R}^d$, let $(b,a)$ be their inner product. If there exists $a\in\mathbb{R}^d$ such that $(x,a)>0$ for all $x\in \con(E) \setminus\{\textbf{0}\}$ then we say that the cone $\con(E)$ is \textit{pointed}. In this case we will say that $C = \con(E)$ is \textit{pointed at $a$}.

If $C\subset \mathbb{R}^d$ is a cone pointed at a vector $a\in\mathbb{R}^d$, and $\beta\geqslant 0$, then we call
\begin{equation*}
    H_{\beta} := \{x\in\mathbb{R}^d\mid (x,a)<\beta\}
\end{equation*}a \textit{truncating half-space} for $C$.

This truncating half-space is a useful concept in describing multiplicities associated to ideals and $p$-families, as we will see later in this section. In particular we will calculate multiplicities by taking volumes of suitable cones intersected with a truncating half space.

\subsection{$p$-systems and $p$-families}\label{sec4}

Now we define families of ideals called $p$-families. We then define sets called $p$-systems which will be useful in measuring the generalized Hilbert-Kunz multiplicities of $p$-families.

\begin{definition}
    Let $R$ be a ring of prime characteristic $p>0$. A sequence of ideals $I_q$ indexed by $q\in \{p^e\mid e\in\mathbb{N}\}$ is called a \textit{$p$-family} if $I_q^{[p]}\subset I_{pq}$ for all $q$.
\end{definition}

For instance, the Frobenius powers $\{I^{[q]}\}$ of an ideal $I$ in $R$ form a $p$-family. Now we define $p$-systems.

\begin{definition}
    Let $S$ be a semigroup and let $p>0$ be a prime number. A collection $T_{\bullet} = \{T_q\}_{q\in\{p^e\mid e\in\mathbb{N}\}}$ of subsets $T_q$ of $S$ is called a $p$-\textit{system} if $T_q + S\subset T_q$ and $pT_q\subset T_{pq}$ for all $q$. If $T_{\bullet} = \{T_q\}$ is a $p$-system, we associate its \textit{$p$-body}, which is the set
    \begin{equation*}
        \Delta(S,T_{\bullet}): = \bigcup_{q=1}^{\infty}\Big(\frac{1}{q}T_q + \con(S)\Big).
    \end{equation*}
\end{definition}

We will see in Theorem \ref{cone1} that certain limits, namely multiplicities as we will show, can be written as the volume of a truncated $p$-body, which is a $p$-body intersected with any truncating half-space.

\subsection{$OK$-Domain}\label{sec5}

We review the construction of $OK$-valuations, $OK$-domains following work of Cutkosky, Jeffries, and Hernández. First we define $a$-valuations for a vector $a\in\mathbb{R}^d$.

\begin{definition}(Remark 2.4 \cite{D1})
    Suppose that $i:\mathbb{Z}^d\to \mathbb{R}$ is an injective group homomorphism. Then there exists a unique $a\in\mathbb{R}^d$ whose coordinates are rationally independent having the property that $i(x) = (x,a)$ for all $x\in\mathbb{Z}^d$. We say that $x\leqslant_a y$ if $x,y\in\mathbb{Z}^d$ and $(x,a)\leqslant (y,a)$.

    Let $K$ be a field and let $K^{\times} = K\setminus\{0\}$. Let $i:\mathbb{Z}^d\to \mathbb{R}$ be an embedding given by a vector $a\in\mathbb{R}^d$. An $a$-\textit{valuation} on $K$ is a surjective group homomorphism $\nu:K^{\times}\to \mathbb{Z}^d$ such that
    \begin{equation*}
        \nu(x+y)\geqslant_a \min\{\nu(x),\nu(y)\}
    \end{equation*}for $x,y\in K^{\times}$ (we formally set $\nu(0):=\infty$). For $M\subset K$, we write $\nu(M) := \nu(M\setminus \{0\})\subset \mathbb{Z}^d$, which is called the \textit{image of $M$} under $\nu$. For $u\in \mathbb{Z}^d$ we define $K_u:= \{x\in K^{\times} \mid \nu(x)\geqslant_a u\}$. The ideals $K_u\cap R$ are called \textit{valuation ideals}.
\end{definition}

Now we describe what it means for a valuation ring to strongly dominate a local domain.

\begin{definition}(Definition 2.7 \cite{D1})
Let $(R,m,k)$ be a $d$-dimensional local domain with fraction field $K$, let $\textbf{a}\in\mathbb{R}^d$, and let $\nu:K^{\times}\to \mathbb{Z}^d$ be an $\textbf{a}$-valuation with value group $\mathbb{Z}^d$. We say that $R$ \textit{strongly dominated} by $V_{\nu}$ if $R$ is dominated by $V_{\nu}$ and $(\textbf{a},\textbf{u})>0$ for all $\textbf{u}\in \con(S)\setminus\{0\}$.
\end{definition}

Let $a\in\mathbb{R}^d$. One special kind of $a$-valuation is that of an $OK$-vaulation, which will play a central role in our proof of the technical theorem of Section \ref{sec6}. $OK$-vaulations are defined as follows.

\begin{definition}(Definition 3.1 \cite{D1})
    Let $(R,m,k)$ be a $d$-dimensional local domain with fraction field $K$. Fix an embedding $\mathbb{Z}^d\hookrightarrow \mathbb{R}$ defined by a vector $a\in\mathbb{R}^d$ and suppose that $\nu:K\to\mathbb{Z}^d$ is an $a$-valuation with associated valuation ring $(V_{\nu},m_{\nu},k_{\nu})$. If
    \begin{enumerate}
        \item $R$ is strongly dominated by $V_{\nu}$
        \item the resulting extension of residue fields $k\hookrightarrow k_{\nu}$ is finite, and
        \item there exists a point $ \textbf{h} \in\mathbb{Z}^d$ s.t.
        \begin{equation*}
            R\cap K_{\geqslant n\textbf{h}}\subset m^n\text{, for all } n\in\mathbb{N}
        \end{equation*}then we say that $(V_{\nu},m_{\nu},k_{\nu})$ is \textit{$OK$-relative to $R$}.
    \end{enumerate}
\end{definition}

\begin{definition}(Definition 3.2 \cite{D1})
    Let $R$ be a $d$-dimensional local domain with fraction field $K$. If there exists a valuation $\nu$ in $K$ with a value group $\mathbb{Z}^d$ whose valuation ring $(V_{\nu},m_{\nu},k_{\nu})$ is $OK$ relative to $R$, then we say that $R$ is an \textit{$OK$-domain}.
\end{definition}

The next result is crucial in the calculations of Section \ref{sec6} leading to our major theorems.

\begin{example}\label{ex1}
(Corollary 3.7 \cite{HJ})
    An excellent local domain that contains a field is an $OK$-domain. Consequently, any complete local domain that contains a field is also an $OK$-domain.
\end{example}

Suppose that $(R,m,k)$ is a $d$-dimensional local domain of characteristic $p>0$ with fraction field $K$. Fix a $\mathbb{Z}$-linear embedding of $\mathbb{Z}^d$ into $\mathbb{R}$ induced by a vector $a\in \mathbb{R}^d$, and a valuation $\nu:K^{\times}\to \mathbb{Z}^d$ that is $OK$-relative to $R$. Let $S$ be the semigroup $\nu(R)\subset \mathbb{Z}^d$, and let $C$ denote the closed cone in $\mathbb{R}^d$ generated by $S$.

If $M$ is a $R$-submodule of $F$, then we see that the $R$-module structure on $M$ induces a $k$-vector space structure on $\frac{M\cap K_{\geqslant u}}{M\cap K_{>u}}$ for $u\in\mathbb{Z}^d$. By definition, this vector space is nonzero if and only if there exists $x\in M$ such that $\nu(x)=u$.

Now we give two further definitions and some theorems which will be useful in the proofs of Section \ref{sec6}.

\begin{definition}\label{dfD13.7}
(Definition 3.7 \cite{D1})
    For a $R$-submodule $M$ of $K$, we define
    \begin{equation*}
        \nu^{(h)}(M)
        = \bigg\{u\in\mathbb{Z}^d\bigg{|}\dim_k\bigg(\frac{M\cap K_{\geqslant u}}{M\cap K_{>u}}\bigg)\geqslant h\bigg\}
    \end{equation*}for $1\leqslant h\leqslant [k_{\nu}:k]$.
\end{definition}

\begin{definition}(Definition 3.10 \cite{D1})
    A semigroup $S\subset\mathbb{Z}^d$ is called \textit{standard} if $S- S = \mathbb{Z}^d$ and $\con(S)$ is pointed. If $R$ is a complete $OK$ domain, with $OK$ valuation $v$, we have that $S:= \nu(D)$ is a standard semigroup.
\end{definition}

Now we mention three results which will be useful in the calculations of Section \ref{sec6}.

\begin{theorem}\label{cone1} (Theorem 4.10 \cite{HJ})
    For a standard semigroup $S$ in $\mathbb{Z}^d$, a $p$-system $T_{\bullet}$ in $S$, and a truncating halfspace $H$ for $\con(S)$, we have
    \begin{equation*}
        \lim_{q\to\infty}\frac{\#(T_q\cap qH)}{q^d}
        = \vol_{\mathbb{R}^d}(\Delta(S,T_{\bullet})\cap H).
    \end{equation*}
\end{theorem}

\begin{theorem}\label{D1thm3.18}
(Theorem 3.18 \cite{D1})
Let $S$ be a standard semigroup in $\mathbb{Z}^d$. If $T_{\bullet}$ is a $p$-system of ideals in $S$ and $H$ is any truncating halfspace for $\con(S)$, then for any given $\epsilon>0$, there exists $q_0$ such that if $q\geqslant q_0$, we have
\begin{equation*}
    \lim_{e\to\infty}\frac{\#((p^eT_q+S)\cap p^eqH)}{p^{ed}q^d}
    \geqslant
    \vol_{\mathbb{R}^d}(\Delta(S,T_{\bullet})\cap H)
    -\epsilon.
\end{equation*}
\end{theorem}

\begin{lemma}\label{lem-new5.10}
(Corollary 5.10 \cite{HJ}, Proposition 4.8 \cite{D1}) For any $p$-family $F_{\bullet}$ in $R$, we have
\begin{equation*}
    \lim_{q\to\infty}\frac{\#(\nu^{(h)}(F_q)\cap qH)}{q^d}
    = \vol_{\mathbb{R}^d}(\Delta(S,\nu(F_{\bullet}))\cap H)
\end{equation*}for $1\leqslant h\leqslant [k_v:k]$, and any truncating halfspace $H$ of $C$.
    %for all h
    %V in HJ 5.10 must be the valuation ring???
\end{lemma}

\section{Generalized Hilbert-Kunz multiplicities and existence conditions}\label{sec2}

In this section we define generalized Hilbert-Kunz multiplicity and we define and discuss the conditions LC and $p(c)$ on an ideal in a local ring which are important for showing the existence of this multiplicity. The condition $p(c)$ will play an important role in future Section \ref{sec7}. The condition $p(c)$ implies the condition LC, and the condition LC is used to prove existence of generalized Hilbert-Kunz multiplicities. All rings $R$ in this paper are Noetherian and  of prime characteristic $p>0$. We begin by defining the generalized Hilbert-Kunz multiplicity of an ideal, after which we will discuss how its existence is obtained from the conditions $p(c)$ and (LC) below

\begin{definition}\label{ghk df}(Epstein, Yao \cite{EY})
    Let $(R,m)$ be a Noetherian local ring and let $I\subset R$ be an ideal such that $\ell_R(H^0_m(R/I^{[q]}))$ is finite for all $q=p^e$. The \textit{generalized Hilbert-Kunz multiplicity} of $I$ is defined as
    \begin{equation*}
        e_{_{g}HK}(I)=e_{_{g}HK}(I,R)= \limsup_{p^e=q\to\infty}\frac{\ell_R(H^0_m(R/I^{[q]}))}{q^d}.
    \end{equation*}We have $(I^{[q]})^{\sat}/I^{[q]} = H^0_m(R/I^{[q]})$ for all $q$, where the saturation $J^{\sat}$ of an ideal $J$ is $J^{\sat}:= J:m^{\infty} =\displaystyle \bigcup_{i\geq 1}J:m^i$.
\end{definition}

Now we discuss the conditions $p(c)$ and (LC) on an ideal $I$ which have been historically important in the study of generalized Hilbert-Kunz multiplicities. Consider the following question.

\begin{question}\label{q1} Suppose that $I$ is an ideal in a Noetherian local ring $(R,m)$ of prime characteristic $p>0$. Does there exist $c\in\mathbb{Z}_{>0}$ such that for all powers $q=p^e$, we have
    \begin{equation}\label{p(c) eq}
        I^{[q]}\cap m^{qc} = (I^{[q]})^{\sat}\cap m^{qc}?
    \end{equation}
    \end{question}

If $I$ satisfies the condition above for some $c>0$, then we say that $I$ satisfies condition $p(c)$. Note that if $I$ is an ideal which satisfies $p(c)$, then $I$ satisfies the condition (LC), which states that
\begin{equation}
I^{\sat}m^{cq}\subset I^{[q]}.
\end{equation}
Conversely, it was shown  (see Lemma 6.6 \cite{HJ}) that if $I\subset R$ is an ideal satisfying (LC) and $\dim(R/I) = 1$, then $I$ actually satisfies condition $p(c)$. That is, the conditions $p(c)$ and (LC) on an ideal $I$ are equivalent when $\dim(R/I)=1$. It is not generally known when an ideal satisfies either condition $p(c)$ or (LC). Vraciu proved that if every ideal satisfies (LC), then $e_{_gHK}(I)$ exists for all ideals $I\subset R$. In  \cite[Proposition 6.2]{HJ}, Jeffries and Hernández proved that for an ideal $I\subset R$, $e_{_gHK}(I)$ exists as long as that one particular ideal satisfies (LC). In Section \ref{sec7}, we show that if an ideal $I$ satisfies condition $p(c)$, then $e_{_gHK}(I)$ is a limit of Amao-type multiplicities, see Theorem \ref{thm6-1} for a precise statement. Amao-type multiplicities are defined in Section \ref{sec7}.

The condition (LC) was also studied by Hochster and Hunecke in relation to the (now known to be false) question of whether tight closure localizes. Vraciu showed in Theorem 2.1 \cite{V1} that (LC), together with uncountable prime avoidance implies that generalized Hilbert-Kunz function of an ideal $I$ is a linear combinations of Hilbert-kunz functions (in a way that does not depend on $I$).

Now we turn to discussing how this assumption works for a $p$-family of ideals in a local ring of characteristic $p$, and what results we get from assuming it.

\begin{definition}(c.f. \cite{CS})\label{p(c)} Let $(R,m)$ be a Noetherian local ring of prime characteristic $p>0$. We will say that that a $p$-family $\cal{F}:=\{I_q\}$ of ideals in $R$ \textit{satisfies the condition $p(c)$} if there exists $c\in\mathbb{Z}_{>0}$ such that for all powers $q=p^e$, we have
    \begin{equation}
        I_q\cap m^{qc} = I_q^{\sat}\cap m^{qc}.
    \end{equation}     In particular, an ideal $I$ in $R$ satisfies the condition $p(c)$ if  the $p$-family $I_{\bullet}=\{I^{[q]}\}$ satisfies it.
\end{definition}

 The condition $p(c)$ on a $p$-family $\{I_q\}$ implies the condition (LC) of \cite{V1} on that $p$-family, which states that $I_q^{\sat}m^{cq}\subset I_q$. If an ideal $I$ satisfies $p(c)$ for some $c$, then $\ell_R((I^{[q]})^{\sat}/I^{[q]})$ is finite (for all $q$), so that the generalized HK multiplicity of $I$ is well-defined. Before stating how we will make use of this condition, we first state a slightly more general condition, associated to two $p$-families.

\begin{definition}\label{gendef}
    Let $(R,m)$ be a Noetherian local ring. Suppose that $I_{\bullet} = \{I_q\}\subset J_{\bullet} = \{J_q\}$ are $p$-families in $R$. We will say that \textit{the inclusion} $I_{\bullet}\subset J_{\bullet}$ \textit{satisfies property $p(c)$} if there exists $c\in\mathbb{Z}_{>0}$ such that
    \begin{equation*}
        I_q\cap m^{cq} = J_q\cap m^{cq}
    \end{equation*}for all $q = p^e$ ($e\geqslant 1$).
\end{definition}

 In Theorem 5.14 \cite{HJ} they show that if $R$ is a complete reduced local ring and $I_{\bullet}\subset J_{\bullet}$ are $p$-families such that the inclusion satisfies property $p(c)$ for some $c>0$. Then the limit
\begin{equation*}
    \lim_{q\to\infty} \frac{\ell_R(J_q/I_q)}{q^d}
\end{equation*}exists.
In Section \ref{sec7} we also show that if $I_{\bullet} = \{I_q\}$ is a $p$-family satisfying condition $p(c)$ then the generalized Hilbert-Kunz multiplicity of the family $I_{\bullet}$ is a limit of Amao-type multiplicities. See Theorem \ref{thm6-2} for a precise statement.

\section{A general volume formula}\label{sec6}

In this section we will prove a general theorem from which a volume equals multiplicity formula for the generalized Hilbert-Kunz multiplicity in Theorem \ref{thm6-1} follows. To accomplish this, we will define some machinery associated to $p$-families. For a Noetherian ring $R$ we denote the set of minimal primes $P$ of $R$ for which $\dim(R/P) = \dim(R)$ by $\ash(R)$.

\begin{theorem}\label{generalmult}
Suppose that $R$ is an analytically unramified local ring of dimension $d$. Let that $I_{\bullet} = \{I_q\}\subset J_{\bullet} = \{J_q\}$ be $p$-families. Suppose further that, for each $q' = p^e$, there exists $p$-families $I_{q',\bullet} = \{I_{q',q}\}\subset J_{q',\bullet} = \{J_{q',q}\}$ such that $J_{q',1} = J_{q'}$, $J_{q',q}\subset J_{qq'}$, $I_{q',1} = I_{q'}$, and $I_{q',q}\subset I_{qq'}$ for all $q$. 
Suppose that there is a positive number $c$ such that
\begin{equation*}
    I_{q',q}\cap m_R^{cqq'} = J_{q',q}\cap m_R^{cqq'}
\end{equation*}for all $q,q'$. Suppose further that whenever $P\in\ash(R)$ such that $P\supset I_1$, we have $P\supset I_q$ for all $q$. Let
\begin{equation*}
    \cal{F}(I_{\bullet},J_{\bullet}):= \lim_{q\to\infty}\frac{\ell_R(J_q/I_q)}{q^d}.
\end{equation*}
Then the above limit exists. Moreover,
\begin{equation*}
    \cal{G}(J_{q',\bullet}, I_{q',\bullet}):=\lim_{q\to\infty}
    \frac{\ell_R(J_{q',q}/I_{q',q})}{q^d}
\end{equation*}exists for all $q'$ and the limit
\begin{equation*}
    \lim_{q'\to\infty}\frac{\cal{G}(J_{q',\bullet}, I_{q',\bullet})}{(q')^d}
\end{equation*}exists, and we have the following formula
\begin{equation}
    \cal{F}(I_{\bullet},J_{\bullet})
    =\lim_{q'\to\infty}\frac{\cal{G}(J_{q',\bullet}, I_{q',\bullet})}{(q')^d}.
\end{equation}
\end{theorem}

\begin{remark}\label{clrmk4.2}
    Suppose that $P\in\ash(R)$ and $I_1\subset P$. Then $I_q,J_q\subset P$ for all $q$ and $I_{q',q},J_{q',q}\subset P$ for all $q',q$.
\end{remark}

\begin{remark}Suppose that $R$ is reduced.
    If $I_1\neq 0$, then $J_q,I_{q',q},J_{q',q}\neq 0$ for all $q',q$.
\end{remark}
\begin{proof}
    Let $a\in I_1\setminus \{0\}$. We have $0\neq a^q\in I_1^q\subset I_q\subset J_q$ so that $I_q,J_q\neq 0$. Moreover, $0\neq b:=a^{q'}\in I_1^{q'}\subset I_{q'} = I_{q',1}$. Consequently, $0\neq b^q\in I_{q',1}^q\subset I_{q',q}\subset J_{q',q}$ so that $I_{q',q},J_{q',q}\neq 0$.
\end{proof}

In order to prove Theorem \ref{generalmult}, we will first reduce to the case that $R$ is complete. Since $R\to \widehat{R}$ is flat, we see that
\begin{equation*}
    \begin{split}
        I_{q',q}\widehat{R}\cap m_{\widehat{R}}^{cq'q}
    &= I_{q',q}\widehat{R}\cap m_R^{cq'q}\widehat{R}
    = (I_{q',q}\cap m_R^{cq'q})\widehat{R}
    \\&= (J_{q',q}\cap m_R^{cq'q})\widehat{R}
    = J_{q',q}\widehat{R}\cap m_{\widehat{R}}^{cq'q}
    \end{split}
\end{equation*}
We will show that if $Q\in \ash(\widehat{R})$ and $I_1\widehat{R}\subset Q$, then $I_q\widehat{R}\subset Q$ for all $q$. We have that $I_1\subset (I_1\widehat{R})\cap R\subset Q\cap R$ and $Q\cap R\in \spec(R)$. Since $(Q\cap R)\widehat{R} \subset Q$, we have
\begin{equation*}
    \begin{split}
        d = \dim(R)&\geqslant \dim(R/Q\cap R)
    =\dim(\widehat{R/Q\cap R}) = \dim(\widehat{R}/((Q\cap R)\widehat{R})
    \\&\geqslant \dim\widehat{R}/Q = d
    \end{split}
\end{equation*}since $(Q\cap R)\widehat{R}\subset Q$, which shows that $\dim(R/Q\cap R) = d$. Now by hypothesis, $I_q\subset Q\cap R$ for all $q$ so that
\begin{equation*}
    I_q\widehat{R}\subset (Q\cap R)\widehat{R}\subset Q
\end{equation*}for all $q$. Hence, the hypothesis of Theorem \ref{generalmult} hold for the $p$-family $\widehat{I_{\bullet}}:= \{I_q\widehat{R}\}$ in the ring $\widehat{R}$.

Since $m_R^{cqq'}(J_{q',q}/I_{q',q})=0$, we have that
$\ell_{\widehat{R}}(\widehat{J_{q',q}}/\widehat{I_{q',q}})$ for all $q$ and $q'$ (since $R\to\widehat{R}$ is local and faithfully flat, we have that $\ell_{\widehat{R}}(M) = \ell_R(M)$ for any $\widehat{R}$-module $M$). Therefore we may assume that $R$ is complete (in the statement of Theorem \ref{generalmult}) by replacing $R$, $J_q$, $I_q$, $I_{q',q}$, and $J_{q',q}$ with $\widehat{R}$, $J_q\widehat{R}$, $I_q\widehat{R}$, $I_{q',q}\widehat{R}$, and $J_{q',q}\widehat{R}$ for all $q$ and $q'$.

We will now prove Theorem \ref{generalmult} in the case that $R$ is complete domain. Now we will assume that $(R,m)$ is a $d$-dimensional complete local domain (of characteristic $p>0$). By Corollary 3.7 \cite{HJ}, $R$ is an $OK$ domain. Let $K$ be a fraction field of $R$. Since $R$ is $OK$, there is a valuation $\nu:K^{\times}\to \mathbb{Z}^d$ which is $OK$-relative to $R$, and let $(V_{\nu},m_{\nu},k_{\nu})$ be the valuation ring of $v$.\\

If $I_1=0$, then $I_q,J_q,I_{q',q},J_{q',q}=0$ for all $q,q'$ by Remark \ref{clrmk4.2}. Thus Theorem \ref{generalmult} holds in this case. Consequently we may assume that $I_1\neq 0$, so that $J_q,I(q')_1,J_{q',q}\neq 0$ for all $q',q$.

Let $S = \nu(R)\subset \mathbb{Z}^d$ be the valuation semigroup and let $C = \con(S)$ denote the closed cone in $\mathbb{R}^d$ generated by $S$.

%make the v's all into \nu's eventually or vice versa
%maybe remove the p system version later esp if its not needed in light of the following lemma.

By assumption,
\begin{equation}\label{eq7.0}
    J_q\cap m^{cq} = I_q\cap m^{cq}
\end{equation}for all $q$.

Since $\nu$ is $OK$ relative to $R$, there exists $v\in\mathbb{Z}^d$ such that $R\cap K_{\geqslant nv}\subset m^n$ for $n\in\mathbb{N}$. Let $w = cv$. Now by (\ref{eq7.0})
\begin{equation}\label{eq7.6}
    I_q \cap K_{\geqslant qw} = J_q\cap K_{\geqslant qw}
\end{equation}  for all $q = p^e$.

For all $q = p^e$ and $u\in\mathbb{Z}^d$, we have the following two exact sequences.
\begin{equation}\label{eq7.3}
    0\to \frac{I_q}{I_q\cap K_{\geqslant u}}\to \frac{J_q}{I_q\cap K_{\geqslant u}}\to \frac{J_q}{I_q}\to 0
\end{equation}
\begin{equation}\label{eq7.4}
    0\to \frac{J_q\cap K_{\geqslant u}}{I_q\cap K_{\geqslant u}}\to \frac{J_q}{I_q\cap K_{\geqslant u}}\to \frac{J_q}{J_q \cap K_{\geqslant u}}\to 0
\end{equation}

Consequently

\begin{equation}\label{eq7.5}
    \begin{split}
        \ell_R(J_q/I_q)
        &= \ell_R(J_q/I_q\cap K_{\geqslant u})
        -\ell_R(I_q/I_q\cap K_{\geqslant u})\\
        &= \ell_R(J_q/J_q\cap K_{\geqslant u})
        + \ell_R(J_q\cap K_{\geqslant u}/I_q\cap K_{\geqslant u})
        - \ell_R(I_q/I_q\cap K_{\geqslant u}).
    \end{split}
\end{equation}

Let
\begin{equation*}
    H = \{ u'\in\mathbb{Z}^d\mid (a,u')<(a,w)\}.
\end{equation*}which is a truncating halfspace for $S$. Let $t = [k_{\nu}:k]$. Now by equations (\ref{eq7.5}), (\ref{eq7.6}) and Lemma 3.9 \cite{D1}, we have for all $q$,
\begin{equation}\label{eq7.7}
    \begin{split}
        \ell_R(J_q/I_q)
        &= \ell_R\bigg(\frac{J_q}{J_q\cap K_{\geqslant qw}}\bigg)
        +\ell_R\bigg(\frac{J_q\cap K_{\geqslant qw}}{I_q\cap K_{\geqslant qw}}\bigg)
        - \ell_R\bigg(\frac{I_q}{I_q\cap K_{\geqslant u}}\bigg)\\
        &=
        \sum_{h=1}^t\# (\nu^{(h)}(J_q)\cap qH)
        - \sum_{h=1}^t \# (\nu^{(h)}(I_q)\cap qH).
    \end{split}
\end{equation}
Similarly, we have 
\begin{equation}\label{eq7.8}
    \begin{split}
        \ell_R(J_{q',q}/I_{q',q})
        &= \ell_R\bigg(\frac{J_{q',q}}{J_{q',q}\cap K_{\geqslant qq'w}}\bigg)
        +\ell_R\bigg(\frac{J_{q',q}\cap K_{\geqslant qq'w}}{I_{q',q}\cap K_{\geqslant qq'w}}\bigg)
        - \ell_R\bigg(\frac{I_{q',q}}{I_{q',q}\cap K_{\geqslant qq'w}}\bigg)\\
        &=
        \sum_{h=1}^t\# (\nu^{(h)}(J_{q',q})\cap qq'H)
        - \sum_{h=1}^t \# (\nu^{(h)}(I_{q',q})\cap qq'H)
    \end{split}
\end{equation}for all $q',q$.

For all $q'$, we have that
\begin{equation*}
        q'H= \{q'u'\in\mathbb{R}^d\mid (a,q'u')<(a,q'w)\}
    =  \{u\in\mathbb{R}^d\mid (a,u)<(a,q'w)\}
    \end{equation*}
so that $q'H$ is a truncating halfspace of $S$. By Lemma \ref{lem-new5.10} we have the following equations.

\begin{equation}\label{eq8.1}
    \begin{split}
        \lim_{q\to\infty}\frac{\#(\nu^{(h)}(I_q)\cap qH)}{q^d}
        &= \vol_{\mathbb{R}^d}(\Delta(S,\nu(I_{\bullet}))\cap H)\\
         \lim_{q\to\infty}\frac{\#(\nu^{(h)}(J_q)\cap qH)}{q^d}
        &= \vol_{\mathbb{R}^d}(\Delta(S,\nu(J_{\bullet}))\cap H)\\
        \lim_{q\to\infty}\frac{\#(\nu^{(h)}(I_{q',q})\cap q(q'H))}{q^d}
        &= \vol_{\mathbb{R}^d}(\Delta(S,\nu(I_{q',\bullet}))\cap q'H)\\
        \lim_{q\to\infty}\frac{\#(\nu^{(h)}(J_{q',q})\cap q(q'H))}{q^d}
        &= \vol_{\mathbb{R}^d}(\Delta(S,\nu(J_{q',\bullet}))\cap q'H)
    \end{split}
\end{equation}for all $q'$.\\

Let $T_q = \nu(J_q)$ ($= \nu(J_q)+S$), $E_q = \nu(I_q)$, $T_{q',q} = \nu(J_{q',q})$ and $E_{q',q} = \nu(I_{q',q})$ for all $q$ and $q'$. We see that $T_{\bullet} = \{T_q\}$, $E_{\bullet} = \{E_q\}$, $T_{q',\bullet} = \{T_{q',q}\}$, and $E_{q',\bullet} = \{E_{q',q}\}$ are all $p$-systems in $S$. Also write $\nu(J_{\bullet}) := \{\nu(J_q)\} = T_{\bullet}$. By Definition 3.12 \cite{D1}, we have that
\begin{equation}\label{eq8.2}
    \begin{split}
        \Delta(S,\nu(J_{\bullet}))
    &= \bigcup_{q=1}^{\infty}\Big(\frac{1}{q}\nu(J_q)+\con(S)\Big)\\
    \Delta(S,\nu(J_{q',\bullet}))
    &= \bigcup_{q=1}^{\infty}\Big(\frac{1}{q}\nu(J_{q',q})+\con(S)\Big).
    \end{split}
\end{equation}On the other hand, we have that
\begin{equation}\label{eq8.3}
    \frac{1}{q}\nu(J_{q',q})
    \subset \frac{1}{q}\nu(J_{qq'})
    \subset \frac{1}{qq'}\nu(J_{qq'}).
\end{equation}Equations (\ref{eq8.2}) (\ref{eq8.3}) together yield that
\begin{equation}\label{eq8.4}
    \Delta(S,\nu(J_{q',\bullet}))
    \subset \bigcup_{q\geqslant 1, q'\mid q}\Big( \frac{1}{q}\nu(J_q)+\con(S)\Big)
    \subset \Delta(S,\nu(J_{\bullet})).
\end{equation}Thus
\begin{equation}\label{eq8.5}
    \vol_{\mathbb{R}^d}(\Delta(S,\nu(J_{q',\bullet}))\cap qH)
    \leqslant \vol_{\mathbb{R}^d}(\Delta(S,\nu(J_{\bullet}))\cap qH)
\end{equation}for all $q$ and $q'$. We also have that
\begin{equation}\label{eq8.6}
    q\nu(J_{q'})+ S = q\nu(J_{q',q})+\nu(R)
    = \{\nu(hf^q)\mid f\in J_{q',1}\text{, }h\in R\}
    \subset \nu(J_{q',1}^{[q]})\subset \nu(J_{q',q})
\end{equation}for all $q$ and $q'$, as in,
\begin{equation*}
    q T_{q'}+S \subset \nu(J_{q',q}).
\end{equation*}Fix $\epsilon>0$. By Theorem \ref{D1thm3.18}, there exists $q_0$ such that for $q'\geqslant q_0$ we have
\begin{equation*}
 \vol_{\mathbb{R}^d}(\Delta(S,\nu(J_{\bullet}))\cap H)-\epsilon
    =\vol_{\mathbb{R}^d}(\Delta(S,T_{\bullet})\cap H)-\epsilon
    \leqslant \lim_{q\to\infty}
    \frac{\#[(qT_{q'}+S)\cap (qq'H)]}{q^d(q')^d}.
\end{equation*}This combined with (\ref{eq8.6}) gives that
\begin{equation*}
    \begin{split}
        \vol_{\mathbb{R}^d}(\Delta(S,\nu(J_{\bullet}))\cap H)-\epsilon
    &\leqslant
    \frac{1}{(q')^d}
    \lim_{q\to\infty}
    \frac{\#[\nu(J_{q',q})\cap q(q'H)]}{q^d}\\
    &=
    \frac{1}{(q')^d}
    \lim_{q\to\infty}
    \frac{\#[\nu^{(1)}(J_{q',q})\cap q(q'H)]}{q^d}
    \end{split}
\end{equation*}for $q'\geqslant q_0$. Then by Lemma \ref{lem-new5.10} and Definition \ref{dfD13.7} we have that
\begin{equation}
    \lim_{q\to\infty}\frac{\#[\nu^{(h)}(J_q)\cap qH]}{q^d} -\epsilon
    = \frac{1}{(q')^d}
    \lim_{q\to\infty}
    \frac{\#[\nu^{(h)}(J_{q',q})\cap qq'H]}{q^d}
\end{equation}for $1\leqslant h\leqslant [k_{\nu}:k]$. Now letting $q'\to\infty$ yields that
\begin{equation}\label{eq8.7}
    \lim_{q\to\infty}\frac{\#[\nu^{(h)}(J_q)\cap qH]}{q^d}
    =\lim_{q'\to\infty} \frac{1}{(q')^d}
    \lim_{q\to\infty}
    \frac{\#[\nu^{(h)}(J_{q',q})\cap qq'H]}{q^d}
\end{equation}for $1\leqslant h\leqslant [k_{\nu}:k]$. Adding the equations (\ref{eq8.7}), $h = 1,\ldots,[k_{\nu}:k]$ gives that
\begin{equation}\label{eq8.8}
    \lim_{q\to\infty}\frac{\displaystyle\sum_{h=1}^{[k_{\nu}:k]}\#[\nu^{(h)}(J_q)\cap qH]}{q^d}
    =\lim_{q'\to\infty} \frac{1}{(q')^d}
    \lim_{q\to\infty}
    \frac{\displaystyle\sum_{h=1}^{[k_{\nu}:k]}\#[\nu^{(h)}(J_{q',q})\cap qq'H]}{q^d}.
\end{equation} 

Repeating the above argument, with $J_q, J_{q,q'}, T_{\bullet}, J_{\bullet}$ and $J_{q',\bullet}$ replaced with $I_q, I_{q,q'}, E_{\bullet}, I_{\bullet}$ and $I_{q',\bullet}$ respectively, yields that

\begin{equation}\label{eq8.9}
    \lim_{q\to\infty}\frac{\displaystyle\sum_{h=1}^{[k_{\nu}:k]}\#[\nu^{(h)}(I_q)\cap qH]}{q^d}
    =\lim_{q'\to\infty} \frac{1}{(q')^d}
    \lim_{q\to\infty}
    \frac{\displaystyle\sum_{h=1}^{[k_{\nu}:k]}\#[\nu^{(h)}(I_{q',q})\cap qq'H]}{q^d}.
\end{equation} Subtracting (\ref{eq8.9}) from (\ref{eq8.8}) and applying equations (\ref{eq7.7}) and (\ref{eq7.8}) yields that

\begin{equation}
    \lim_{q'\to\infty}\frac{1}{(q')^d}\lim_{q\to\infty}\frac{\ell_R(J_{q',q}/I_{q',q})}{q^d}
    =
    \lim_{q\to\infty}\frac{\ell_R(J_q/I_q)}{q^d}.
\end{equation}This completes the proof of Theorem \ref{generalmult} in the case that $R$ is a complete domain.

%%%%%%%

%%%%%

%An unram case here

%%%%

%%%its over here guys

%%keep

Now assume that $R$ is only analytically unramified. We will prove Theorem \ref{generalmult} in this case. We may continue to assume that $R$ is complete, and hence reduced. Now we introduce some notation. Let $P_1,\ldots,P_r$ be the minimal primes of $R$. Let $R_i = R/P_i$ for each $i$. Let $T = R_1\oplus \ldots \oplus R_r$ and let $\phi:R\to T$ be the natural inclusion. By the Artin-Rees Lemma, there exists a positive integer $\lambda$ such that
\begin{equation*}
    \omega_n:= \phi^{-1}(m_R^nT) = R\cap (m_R^nT)\subset m_R^{n-\lambda}
\end{equation*}for $n\geqslant \lambda$. Hence, $m_R^n\subset \omega_n\subset m_R^{n-\lambda}$ for $n\geqslant \lambda$. In particular,
\begin{equation*}
    m_R^{mk}\subset \omega_{mk}\subset m_R^{mk-\lambda}
\end{equation*}when $m,k>0$ and $mk\geqslant \lambda$. We also have that
\begin{equation*}
    \omega_n = \phi^{-1}(m_R^nR_1\oplus \cdots \oplus m_R^nR_r)
    = (m_R^n+P_1)\cap \cdots \cap (m_R^n+P_r)
\end{equation*}for all $n$. Let $\beta = (\lambda + 1)c$. We have that $\omega_\beta n\subset m_R^{c(\lambda+1)n-\lambda}\subset m_R^{cn}$ for $n\geqslant 1$. Thus, $\omega_{\beta qq'}\subset m_R^{cqq'}$ for all powers $q$ and $q'$ of $p$. Consequently, we have
\begin{equation*}
    \omega_{\beta qq'}\cap I_{q,q'}
    = \omega_{\beta qq'}\cap (m_R^{cqq'})\cap I_{q,q'})
    = \omega_{\beta qq'}\cap (m_R^{cqq'})\cap J_{q,q'})
    =\omega_{\beta qq'}\cap J_{q,q'}
\end{equation*}for all $q,q'$. In particular, $\omega_{\beta q}\cap I_q=\omega_{\beta q}\cap J_q$ for all $q$. Thus,

\begin{equation}\label{eq7.1}
    \ell_R(J_q/I_q) 
    = \ell_R(J_q/\omega_{\beta q}\cap J_q)
    - \ell_R(I_q/\omega_{\beta q}\cap I_q)
\end{equation} for all $q$ and
\begin{equation}\label{eq7.1'}
    \ell_R(J_{q',q}/I_{q',q})
    = \ell_R(J_{q',q}/\omega_{\beta qq'}\cap J_{q',q})
    - \ell_R(I_{q',q}/\omega_{\beta qq'}\cap I_{q',q})
\end{equation}for all $q,q'$.\\

Define $E_n^0 = R$ and
\begin{equation*}
    E_n^j = (m_R^{\beta n}+P_1)\cap \cdots \cap (m_R^{\beta n}+P_r)
\end{equation*} for $1\leqslant j\leqslant r$ and $n\geqslant 1$. Let
\begin{equation*}
    L_0^j = F_0^j = B_{m,0}^j = A_{m,0}^j = R
\end{equation*}for $0\leqslant j\leqslant r$ and $m\geqslant 1$. Let $L_q^j = E_q^j\cap J_q$, $F_q^j = E_q^j\cap I_q$, $B_{q',q}^j = E_{qq'}^j\cap J_{q',q}$, and $A_{q',q}^j = E_{qq'}^j\cap I_{q',q}$ for $0\leqslant j\leqslant r$ and all $q,q'$. Notice that $B_{q',q}^r = E_{qq'}^r\cap J_{q',q} = \omega_{\beta qq'}\cap J_{q',q}$, $A_{q',q}^r = \omega_{\beta qq'}\cap I_{q',q}$, $L_q^r = \omega_{\beta q}\cap J_q$, and $F_q^r = \omega_{\beta q}\cap I_q$ for all $q, q'$. We have isomorphisms 
\begin{equation}\label{eq7.2}
    L_q^j/L_q^{j+1}
    = L_q^j/[(m_R^{\beta q}+P_{j+1})\cap L_q^j]
    \cong \frac{L_q^j + m_R^{\beta q}+P_{j+1}}{m_R^{\beta q}+P_{j+1}}
    = \frac{(L_q^j + P_{j+1}) + (m_R^{\beta q}+P_{j+1})}{m_R^{\beta q}+P_{j+1}}
\end{equation}so that
\begin{equation*}
\begin{split}
    L_q^jR_{j+1}/[(L_q^jR_{j+1})\cap (m_{R_{j+1}}^{\beta q})]
    &= \frac{(L_q^j+P_{j+1})/P_{j+1}}{[(L_q^j+P_{j+1})/P_{j+1}]\cap [(m_R^{\beta q}+ P_{j+1})/P_{j+1}]}\\
    &= \frac{(L_q^j+P_{j+1})/P_{j+1}}{[(L_q^j+P_{j+1})\cap (m_R^{\beta q}+ P_{j+1})]/P_{j+1}}\\
    &\cong \frac{L_q^j+P_{j+1}}{(L_q^j+P_{j+1})\cap (m_R^{\beta q}+ P_{j+1})}\\
\end{split}
\end{equation*}
and hence
\begin{equation}\label{7.9-1}
    L_q^j/L_q^{j+1}\cong
    L_q^jR_{j+1}/[(L_q^jR_{j+1})\cap (m_{R_{j+1}}^{\beta q})]
\end{equation}for $0\leq j\leq r-1$. Similarly, 
\begin{equation}\label{7.10-1}
    \begin{split}
        F_q^j/F_q^{j+1}&\cong
    F_q^jR_{j+1}/[(F_q^jR_{j+1})\cap (m_{R_{j+1}}^{\beta q})]\\
    B_{q',q}^j/B_{q',q}^j
    &\cong B_{q',q}^jR_{j+1}/
    [(B_{q',q}^jR_{j+1})\cap m_{R_{j+1}}^{\beta qq'}]\\
    A_{q',q}^j/A_{q',q}^j
    &\cong A_{q',q}^jR_{j+1}/
    [(A_{q',q}^jR_{j+1})\cap m_{R_{j+1}}^{\beta qq'}].
    \end{split}
\end{equation}

Consequently, the following equations hold
\begin{equation*}
    \begin{split}
        \ell_R(J_q/\omega_{\beta q}\cap J_q)
        &=\sum_{i=0}^{r-1}\ell_R(L_q^i/L_q^{i+1})\\
        \ell_R(I_q/\omega_{\beta q}\cap I_q)
        &=\sum_{i=0}^{r-1}\ell_R(F_q^i/F_q^{i+1})
    \end{split}
\end{equation*}

\begin{equation*}
\begin{split}
\ell_R(J_{q',q}/\omega_{\beta qq'}\cap J_{q',q})
        &=\sum_{i=0}^{r-1}\ell_R(B_{q',q}^i/B_{q',q}^{i+1})\\
        \ell_R(I_{q',q}/\omega_{\beta qq'}\cap I_{q',q})
        &=\sum_{i=0}^{r-1}\ell_R(A_{q',q}^i/A_{q',q}^{i+1}).
\end{split}
\end{equation*}

Now by (\ref{eq7.1}) and (\ref{eq7.1'}) we have that
\begin{equation}\label{7.11-1}
    \begin{split}
        \ell_R(J_q/I_q) &= \sum_{i=0}^{r-1}(\ell_R(L_q^i/L_q^{i+1})-\ell_R(F_q^i/F_q^{i+1}))\\
    \ell_R(J_{q',q}/I_{q',q}) &= \sum_{i=0}^{r-1}(\ell_{R_i}(\beta_{q',q}^i/\beta_{q',q}^{i+1})-\ell_{R_i}(\alpha_{q',q}^i/\alpha_{q',q}^{i+1})).
    \end{split}
\end{equation}for all $q$.

Equations (\ref{7.9-1}), (\ref{7.10-1}), and (\ref{7.11-1}) together yield that
\begin{equation}\label{eq-main}
    \begin{split}
        \ell_R(J_q/I_q) &=
        \sum_{i=0}^{r-1}
        (\ell_R(L_q^jR_{j+1}/[(L_q^jR_{j+1})\cap (m_{R_{j+1}}^{\beta q})])
        \\&-
        \ell_R(F_q^jR_{j+1}/[(F_q^jR_{j+1})\cap (m_{R_{j+1}}^{\beta q})])\\
        \ell_R(J_{q',q}/I_{q',q}) &=
        \sum_{i=0}^{r-1}
        (\ell_R(\beta_{q',q}^jR_{j+1}/
    [(\beta_{q',q}^jR_{j+1})\cap m_{R_{j+1}}^{\beta qq'}])
        \\&-\ell_R(\alpha_{q',q}^jR_{j+1}/
    [(\alpha_{q',q}^jR_{j+1})\cap m_{R_{j+1}}^{\beta qq'}])).
    \end{split}
\end{equation}

%The following lemma holds by Theorem 5.14 \cite{HJ}.
%\begin{lemma}\label{lem7.2}
%    Let $R$ be a complete domain of characteristic $p>0$ and let $I_{\bullet}=\{I_q\}$ and $J_{\bullet} = \{J_q\}$ be $p$-families in $R$ such that $I_{\bullet}\subset J_{\bullet}$ satisfies $p(c_0)$ for some $c_0>0$. Let $d$ be a positive integer and suppose that $\dim(R)<d$. Then
%    \begin{equation*}
%        \lim_{q\to\infty}\frac{\ell_R(J_q/I_q)}{q^d}
%    \end{equation*}exists and equals zero.
%\end{lemma}

Fix $j\in\{0,\ldots, r-1\}$ and let $\overline{R} = R_{j+1}$. Take $\overline{J}_q = L_q^j\overline{R}$, and $\overline{I}_q = \overline{J}_q\cap m_{\overline{R}}^{\beta q}$, $\overline{J}(q')_q = B_{q',q}^j\overline{R}$, $\overline{I}(q')_q = \overline{J}(q')_q\cap m_{\overline{R}}^{\beta q'q}$. Then $\overline{I}_q\subset \overline{J}_q$  and $\overline{I}(q')_q\subset \overline{J}(q')_q$ for all $q$ and $q'$. Since $R$ is complete, we have that $\overline{R} = R/P_{j+1}$ is a complete domain.

We have families $\overline{I_{\bullet}}(q') := \{\overline{I}(q')_q\}$, $\overline{J_{\bullet}}(q') := \{\overline{J}(q')_q\}$, $\overline{I_{\bullet}}:= \{\overline{I}_q\}$, and $\overline{J_{\bullet}}:=\{\overline{J}_q\}$ of ideals in $\overline{R}$. Now we will prove the following lemma.

\begin{lemma}\label{lem7.1}
    The families of ideals $\overline{I_{\bullet}}(q')$, $\overline{J_{\bullet}}(q')$, $\overline{I_{\bullet}}$, and $\overline{J_{\bullet}}$ are $p$-families in $\overline{R}$ and they satisfy the hypothesis of Theorem \ref{generalmult} and in particular
    \begin{equation}\label{eq7.9}
    \overline{I}(q')_q\cap m_{\overline{R}}^{\beta q'q}
    = \overline{J}(q')_{q}\cap m_{\overline{R}}^{\beta q'q}
\end{equation}for all $q'$ and $q$.
\end{lemma}
\begin{proof}
We will show that the families $\overline{I_{\bullet}}(q')$, $\overline{J_{\bullet}}(q')$, $\overline{I_{\bullet}}$, and $\overline{J_{\bullet}}$ satisfy the hypothesis of Theorem \ref{generalmult}. Fix $q\geq 1$ we have that
    \begin{equation*}
        \overline{J}(q')_q
        =\beta_{q',q}^j\overline{R}
        =(E_{qq'}^j\cap J_{q',q})\overline{R}
        \subset (E_{qq'}^j\cap J_{qq'}\overline{R}
        = L_{qq'}^j\overline{R} = \overline{J}_{qq'}.
    \end{equation*}Similarly, $\overline{I}(q')_q\subset \overline{I}_{qq'}$. We also have
    \begin{equation*}
        \overline{J}(q')_1 = (E_{q'}^j\cap J(q')_1)\overline{R}
        =(E_{q'}^j\cap J_{q'})\overline{R} = \overline{J}_{q'}
    \end{equation*}and similarly, $\overline{I}(q')_1 = \overline{I}_{q'}$.
By definition we have that $\overline{J}(q')_{q}\cap m_{\overline{R}}^{\beta q'q}=\overline{I}(q')_q$ so that
\begin{equation}\label{eq7.9-2}
    \overline{I}(q')_q\cap m_{\overline{R}}^{\beta q'q}
    = \overline{J}(q')_{q}\cap m_{\overline{R}}^{\beta q'q}
\end{equation}for all $q'$ and $q$.
\end{proof}

We will now prove that
\begin{equation}\label{eq7.10}
    \lim_{q'\to\infty}\lim_{q\to\infty}
    \frac{\ell_{\overline{R}}(\overline{J}(q')_{q}/\overline{I}(q')_{q})}{q^d(q')^d}
    =\lim_{q\to\infty}
    \frac{\ell_{\overline{R}}(\overline{J}_q/\overline{I}_q)}{q^d}.
\end{equation}

Fix $q'\geq 1$. Let $c_0 = \beta q'$. We have by (\ref{eq7.9-2}) that
\begin{equation}\label{eq7.11}
    \overline{I}(q')_q\cap m_{\overline{R}}^{c_0q}
    = \overline{J}(q')_{q}\cap m_{\overline{R}}^{c_0q}
\end{equation}for all $q$. Then since $\overline{J}(q)_{1} = \overline{J}_{q}$ and $\overline{I}(q)_{1} = \overline{I}_{q}$ $q$ we (also by (\ref{eq7.9}) ) have that
\begin{equation}\label{eq7.12}
    \overline{J}_q\cap m_{\overline{R}}^{\beta q}
    =\overline{J}_q\cap m_{\overline{R}}^{\beta q}
\end{equation}for all $q$.\\

Suppose that $\dim(\overline{R})<d$. Then by Theorem 5.14 \cite{HJ} and equations (\ref{eq7.11}) and (\ref{eq7.12}) we have that
\begin{equation*}
    \lim_{q'\to\infty}\frac{1}{(q')^d}\lim_{q\to\infty}
    \frac{\ell_{\overline{R}}(\overline{J}(q')_{q}/\overline{I}(q')_{q})}{q^d}
    =0=\lim_{q\to\infty}
    \frac{\ell_{\overline{R}}(\overline{J}_q/\overline{I}_q)}{q^d}.
\end{equation*}

Suppose that $\dim(\overline{R})=d$. We will prove that if $\overline{I}_1=0$, then $\overline{I}_q=0$ for all $q$. Suppose that $\overline{I}_1=0$. If $\overline{J}_1= \overline{R}$ then
\begin{equation*}
    0 = \overline{I}_1\cap m_{\overline{R}}^{\beta}
    = \overline{J}_1\cap m_{\overline{R}}^{\beta}
    =m_{\overline{R}}^{\beta}
\end{equation*}which contradicts the fact that $\dim(\overline{R})\neq 0$. So $\overline{J}_1\neq \overline{R}$. Consequently, $\overline{J}_1^{\beta}\subset \overline{J}_1\cap m_{\overline{R}}^{\beta} = 0$, which implies that $\overline{J}_1=0$, since $\overline{R}$ is reduced. So $0 = \overline{J}_1 = (L_1^j+P_{j+1})/P_{j+1}$, as in, $L_1^j\subset P_{j+1}$. We now have that
\begin{equation*}
    m_R^{\beta}\cap J_1\subset E_1^j\cap J_1 = L_1^j\subset P_{j+1}
\end{equation*}which implies that $J_1\neq R$ (otherwise $m_R^{\beta}$ is contained in the minimal prime $P_{j+1}$, contradicting the fact that $\dim(R)\neq 0$). Now we have that
\begin{equation*}
    J_1^{\beta}\subset m_R^{\beta}\cap J_1\subset P_{j+1}
\end{equation*}so that $I_1\subset J_1\subset P_{j+1}$. Now by Remark \ref{clrmk4.2} we have that $I_q$, $J_q$, $I_{q',q}$, and $J_{q',q}$ are all contained in $P_{j+1}$ for all $q'$ and $q$. Consequently,
\begin{equation*}
    L_q^j= E_q^j\cap J_q
    \subset P_{j+1}
\end{equation*}so that $\overline{I}_q\subset \overline{J}_q = (L_q^j+P_{j+1})/P_{j+1}=0$. From the case of Theorem \ref{generalmult} when $R$ is a complete $d$-dimensional domain, by replacing $R$, $c$, $I_q$, $J_q$, $I_{q',q}$, and $J_{q',q}$ with $\overline{R}$, $\beta$, $\overline{I}_q$, $\overline{J}_q$, $\overline{I}(q')_q$, and $\overline{J}(q')_q$ respectively, we have that
\begin{equation*}
     \lim_{q'\to\infty}\lim_{q\to\infty}
    \frac{\ell_{\overline{R}}(\overline{J}(q')_{q}/\overline{I}(q')_{q})}{q^d(q')^d}
    =\lim_{q\to\infty}
    \frac{\ell_{\overline{R}}(\overline{J}_q/\overline{I}_q)}{q^d}.
\end{equation*}

That is,
\begin{equation}\label{eq7.13}
    \begin{split}
        \lim_{q'\to\infty}\lim_{q\to\infty}&
    \frac{\ell_{R}(\beta_{q',q}^jR_{j+1}/(\beta_{q',q}^jR_{j+1}\cap m_{R_{j+1}}^{\beta q' q}))}{q^d(q')^d}
    \\&=\lim_{q\to\infty}
    \frac{\ell_R(L_q^jR_{j+1}/(L_q^jR_{j+1}\cap m_{R_{j+1}}^{\beta q}))}{q^d}
    \end{split}
\end{equation}for $j\in \{0,\ldots, r-1\}$. By adding the equations (\ref{eq7.13}) over $j=0,\ldots, r-1$, we see that
\begin{equation}\label{eq7.14}
    \begin{split}
        \lim_{q'\to\infty}\lim_{q\to\infty}&
    \frac{\displaystyle\sum_{j=0}^{r-1}\ell_{R}(\beta_{q',q}^jR_{j+1}/(\beta_{q',q}^jR_{j+1}\cap m_{R_{j+1}}^{\beta q' q}))}{q^d(q')^d}
    \\&=\lim_{q\to\infty}
    \frac{\displaystyle\sum_{j=0}^{r-1}\ell_R(L_q^jR_{j+1}/(L_q^jR_{j+1}\cap m_{R_{j+1}}^{\beta q}))}{q^d}.
    \end{split}
\end{equation}By applying the above argument to $\overline{J}_q := F_q^jR_{j+1}$, $\overline{I}_q:= \overline{J}_q\cap m_{R_{j+1}}^{\beta q}$, $\overline{J}(q')_q:= \alpha_{q',q}^jR_{j+1}$, and $\overline{I}(q')_q:= \overline{J}(q')_q\cap m_{R_{j+1}}^{\beta q'q}$, we have that

\begin{equation}\label{eq7.15}
    \begin{split}
        \lim_{q'\to\infty}\lim_{q\to\infty}&
    \frac{\displaystyle\sum_{j=0}^{r-1}\ell_{R}(\alpha_{q',q}^jR_{j+1}/(\alpha_{q',q}^jR_{j+1}\cap m_{R_{j+1}}^{\beta q' q}))}{q^d(q')^d}
    \\&=\lim_{q\to\infty}
    \frac{\displaystyle\sum_{j=0}^{r-1}\ell_R(F_q^jR_{j+1}/(F_q^jR_{j+1}\cap m_{R_{j+1}}^{\beta q}))}{q^d}.
    \end{split}
\end{equation}

Subtracting the equation (\ref{eq7.15}) from (\ref{eq7.14}) then applying (\ref{eq-main}) we obtain
\begin{equation*}
    \lim_{q'\to\infty}\lim_{q\to\infty}
    \frac{\ell_R(\ell_R(J_{q',q}/I_{q',q})}{q^d(q')^d}
    = \lim_{q\to\infty}
    \frac{\ell_R(J_q/I_q)}{q^d}
\end{equation*}so that
\begin{equation*}
    \lim_{q'\to\infty}\frac{\cal{G}(J_{q',\bullet}, I_{q',\bullet})}{(q')^d}
    = \lim_{q'\to\infty}\frac{1}{(q')^d}\lim_{q\to\infty}
    \frac{\ell_R(J_{q',q}/I_{q',q})}{q^d}
    = \lim_{q\to\infty}
    \frac{\ell_R(J_q/I_q)}{q^d}
    = \cal{F}(I_{\bullet},J_{\bullet}).
\end{equation*}This completes the proof of Theorem \ref{generalmult}.

\section{Applications to generalized Hilbert-Kunz multiplicities}\label{sec7}

In this section we will prove a volume equals multiplicity formula for generalized Hilbert-Kunz multiplicities (Theorem \ref{thm6-1}). Throughout this section, $(R,m)$ will be a Noetherian local ring of prime characteristic $p>0$.

Suppose that $I\subset R$ is an ideal. By \cite{A1}, we have that the limit
\begin{equation*}
    a(I, I^{\sat}) := \lim_{n\to\infty}\frac{\ell_R((I^n)^{\sat}/I^n)}{n^d/d!}
\end{equation*}exists and equals a nonnegative integer. $a(I, I^{\sat})$ is called the \textit{Amao multiplicity}. Amao multiplicities are discussed in \cite{A1}, \cite{R1}, and \cite{SH}. To define the characteristic $p$ analogue of Amao multiplicities, we first need the following remark.

\begin{remark}
    Suppose that $I\subset J$ are ideals in $R$ such that $J/I$ has finite length as an $R$-module. Then $\ell_R(J^{[q]}/I^{[q]})$ has finite length for all powers $q$ of $p$.
\end{remark}
\begin{proof}
    There is a positive integer $c$ such that $m^c(J/I) = 0$. After possibly enlarging $c$, we may assume that $c = p^e$ for some $e\geqslant 1$. Let $\mu :=\mu(m)$ be the embedding dimension of $R$ and let $M = J^{[q]}/I^{[q]}$. We will show that $m^{cq\mu}M= 0$. Let $\alpha\in M$. Write $\alpha = \sum_{i=1}^nr_if_i^{q} +I^{[q]}$, where $f_1,\ldots, f_n\in J$ and $r_1,\ldots,r_n\in R$. Suppose that $y_1,\ldots, y_c\in m$. Then
    \begin{equation*}
        (y_1\ldots y_c)^q\sum_{i=1}^nr_if_i^{q}
        = \sum_{i=1}^nr_i(y_1\ldots y_c f_i)^q\in I^{[q]}
    \end{equation*}so that $(m^c)^{[q]}(J^{[q]}/I^{[q]})=0$. Hence,
    \begin{equation*}
        0 = (m^{[c]})^{[q]}M = m^{[cq]}M
        \supset m^{cq\mu}M
    \end{equation*}which finishes the proof that $M$ has finite length, as $M$ is now an $R/m^{cq\mu}$-module.
\end{proof}

\begin{definition}\label{df-at}
    Let $d = \dim(R)$, and let $I$ and $J$ be ideals $R$ such that $I\subset J$ and $J/I$ has finite length.
    We define the \textit{Amao-type} multiplicity to be the limit superior
    \begin{equation*}
        a_F(I,J):= \limsup_{q\to\infty}\frac{\ell_R(J^{[q]}/I^{[q]})}{q^d}.
    \end{equation*}
    If $I\subset R$ is an $m$-primary ideal, and $J = R$, then $a_F(I,J) = e_{HK}(I)$ is the Hilbert-Kunz multiplicity of $I$.
\end{definition}

%some checking needs done more here, I checked all of this for the polynomial ring setting but not 100% on the power series setting yet

\begin{lemma}\label{lem6-2}
    Let $I\subset R$ be an ideal. Then
    \begin{equation*}
        I^{\sat} = I:m^{[q]}
    \end{equation*}for all $q$ sufficiently large.
\end{lemma}
\begin{proof}
    Let $r$ be the minimal number of generators of $m$. We have that
    \begin{equation*}
        m^{qr}\subset m^{[q]}\subset m^q
    \end{equation*} for all $q$. Thus
    \begin{equation}\label{eq9.1}
        I:m^q\subset I:m^{[q]}\subset I:m^{qr}
    \end{equation}for all $q$. Take $t\in\mathbb{Z}_{>0}$ so large that $I^{\sat} = I:m^{t'}$ for $t'\geqslant t$ and let $q = p^{t}$. Since $q,qr > t$, we have that $I:m^q = I^{\sat} = I:m^{qr}$. Now $I:m^{[q]} = I^{\sat}$ by (\ref{eq9.1}).
\end{proof}

\begin{lemma}\label{lem6-3}
    Let $I\subset R$ be an ideal. We have that
    \begin{equation*}
        (I^{\sat})^{[q]}\subset (I^{[q]})^{\sat}
    \end{equation*}for all $q$.
\end{lemma}
\begin{proof}
Fix a power $q$ of $p$. By Lemma \ref{lem6-2} we may find a sufficiently large power $q'$ of $p$ such that
    \begin{equation*}
        I^{\sat}=I:m^{[q']} \text{ and }
        (I^{[q]})^{\sat} = I^{[q]}:m^{[qq']}.
    \end{equation*} Consequently,
    \begin{equation*}
        (I^{\sat})^{[q]}
        \subset I^{[q]}:((m^{[q']})^{[q]})
        = I^{[q]}: m^{[qq']}
        =(I^{[q]})^{\sat}.
    \end{equation*}
\end{proof}

\begin{lemma}\label{lem6-4}
    Suppose that $I\subset R$ is an ideal which satisfies condition $p(c)$ for some $c>0$. Then we have
    \begin{enumerate}
        \item[$(1)$]
        $(I^{\sat})^{[q]}\cap m^{cq}
        = I^{[q]}\cap m^{cq}$ for all $q$, and
        \item[$(2)$]
        $[(I^{[q']})^{\sat}]^{[q]}\cap m^{cqq'}
        = I^{[qq']}\cap m^{cqq'}$ for all $q$ and $q'$.
    \end{enumerate}
\end{lemma}
\begin{proof}
    Fix powers $q$ and $q'$ of $p$. First we prove (1). Combining Lemma \ref{lem6-3} and the assumption that $I$ satisfies $p(c)$ yields that
    \begin{equation*}
        \begin{split}
            (I^{[q]})^{\sat}\cap m^{qc}
            = I^{[q]}\cap m^{qc}
            \subset (I^{\sat})^{[q]}\cap m^{qc}
            \subset (I^{[q]})^{\sat}\cap m^{qc}
        \end{split}
    \end{equation*}so that $(I^{\sat})^{[q]}\cap m^{qc} = I^{[q]}\cap m^{qc}$. We now prove (2). We have that
    \begin{equation*}
        \begin{split}
            (I^{[qq']})^{\sat}\cap m^{cqq'}
            &= I^{[qq']}\cap m^{cqq'}
            = (I^{[q']})^{[q]}\cap m^{cqq'}
            \subset [(I^{[q']})^{\sat}]^{[q]}\cap m^{cqq'}
            \\&\subset [(I^{[q']})^{[q]}]^{\sat}]\cap m^{cqq'}
            = (I^{[qq']})^{\sat}\cap m^{cqq'}
        \end{split}
    \end{equation*}so that $[(I^{[q']})^{\sat}]^{[q]}\cap m^{cqq'}
        = I^{[qq']}\cap m^{cqq'}$, as desired.
\end{proof}

We have the following volume equals multiplicity formula for the generalized Hilbert-Kunz multiplicity, which is the generalization of Theorem 5.3 \cite{D1} to arbitrary ideals in an analytically unramified ring which satisfy $p(c)$ for some positive number $c$.

%The main result in \cite{D1} is the case of Theorem \ref{thm6-1} when $I$ is primary to the maximal ideal.

\begin{theorem}\label{thm6-1}
    Let $(R,m)$ be a $d$-dimensional analytically unramified local ring of prime characteristic $p>0$. Let $I\subset R$ be an ideal which satisfies condition $p(c)$ for some $c>0$. Then
    \begin{enumerate}
        \item[$(i)$]For all $q'$, the limit
        \begin{equation*}
            a_F(I^{[q']}, (I^{[q']})^{\sat}) = \lim_{q\to\infty}\frac{\ell_R([(I^{[q']})^{\sat}]^{[q]}/I^{[qq']})}{q^d}
        \end{equation*}exists.
        \item[$(ii)$] The limit
        \begin{equation*}
            \lim_{q'\to\infty}\frac{a_F(I^{[q']}, (I^{[q']})^{\sat})}{(q')^d}
        \end{equation*}exists.
        \item[$(iii)$] We have the following formula
        \begin{equation*}
            e_{_gHK}(I)
            = \lim_{q'\to\infty}\frac{a_F(I^{[q']}, (I^{[q']})^{\sat})}{(q')^d}.
        \end{equation*}
    \end{enumerate}
\end{theorem}

\begin{proof} First note that if $F\subset R$ is an ideal then $\{F^{[q]}\}$ is a $p$-family. Suppose that $\{F_q\}$ is a $p$-family. We will show that $\{(F_q)^{\sat}\}$ is a $p$-family. By Lemma \ref{lem6-3} we have that
\begin{equation*}
    [(F_q)^{\sat}]^{[p]}
    \subset (F_q^{[p]})^{\sat}
    \subset F_{qp}^{\sat},
\end{equation*}which shows that $\{(F_q)^{\sat}\}$ is a $p$-family. For all powers $q'$ and $q$ of $p$, let $J_{q',q} = [(I^{[q']})^{\sat}]^{[q]}$, $I_{q',q} = I^{[q'q]}$, $J_q = (I^{[q]})^{\sat}$, and $I_q = I^{[q]}$. Then $J_{q',\bullet}: = \{J_{q',q}\}$, $I_{q',\bullet}: = \{I_{q',q}\}$ are $p$-families in $R$ for all $q'$.

We will show that the families $J_{q',\bullet}$, $I_{q',\bullet}$, $J_{\bullet}$, and $I_{\bullet}$ satisfy the hypothesis of Theorem \ref{generalmult}. By Lemma \ref{lem6-4}, we have that
\begin{equation*}
    J_{q',q}\cap m^{cqq'}
    = I_{q',q}\cap m^{cqq'}
\end{equation*}for all $q$ and $q'$. By construction, we have that $J_{q',q} = J_{q'}$, $I_{q',1} = I_{q'}$, $J_{q',q}\subset J_{q'q}$, and $I_{q',q}\subset I_{q'q}$ for all $q$ and $q'$. The conclusions of the theorem now follow from Theorem \ref{generalmult}.
\end{proof}

%Note that if $I\subset R$ is an ideal satisfying condition $p(c)$, then
%\begin{equation*}
    %(I^{[q]}\cap m^{qc})\widehat{R}
    %= ((I^{[q]})^{\sat}\cap m^{qc})\widehat{R}.
%\end{equation*} As in,
%\begin{equation*}
   % (I\widehat{R})^{[q]}\cap m_{\widehat{R}}^{qc}
    %= [(I\widehat{R})^{[q]}]^{\sat}\cap m_{\widehat{R}}^{qc}
%\end{equation*}so that $p(c)$ is also satisfied by the ideal $I\widehat{R}$ in $\widehat{R}$.\\

%Now we give an alternate proof of (i) in Theorem \ref{thm6-1} using Theorem 5.14 \cite{HJ}.\\

%\begin{proof}
   % We have that
   % \begin{equation*}
        %\ell_R([(I^{[q']})^{\sat}]^{[q]}/I^{[qq']})
        %= \ell_{\widehat{R}}([((\widehat{R})^{[q']})^{\sat}]^{[q]}/(\widehat{R})^{[qq']}).
    %\end{equation*}%
   % By the preceding paragraph, the ideal $I\widehat{R}$ of $\widehat{R}$ also satisfies condition $p(c)$. Thus, by replacing $R$ and $I$ with $\widehat{R}$ and $I\widehat{R}$ respectively, we may assume that $R$ is complete, and hence reduced. Moreover, $R/P$ is complete, and hence a complete domain, for all $P\in\spec(R)$. In particular, $R/P$ is a complete $OK$ domain for all minimal primes $P$ of $R$ by Example \ref{ex1}. Now by Theorem 5.14 \cite{HJ} and Lemma \ref{lem6-4} (2), the limit
    %\begin{equation*}
        %\lim_{q\to\infty}\frac{\ell_R([(I^{[q']})^{\sat}]^{[q]}/I^{[qq']})}{q^d}
    %\end{equation*}exists.
%\end{proof}%

\begin{definition}
    Let $R$ be a Noetherian local ring of prime characteristic $p>0$. Let $I_{\bullet} = \{I_q\}$ be a $p$-family in $R$. Call the limit superior
    \begin{equation*}
        e_{_gHK}(I_{\bullet})
        = \limsup_{q\to\infty}\frac{\ell_R(I_q^{\sat}/I_q)}{q^d}
    \end{equation*}the \textit{generalized Hilbert-Kunz multiplicity} of the family $\{I_q\}$.
\end{definition}
We see that if $I_{\bullet} = \{I^{[q]}\}$ is a Frobenius system of an ideal $I\subset R$, then $e_{_gHK}(I_{\bullet}) = e_{_gHK}(I)$ is the generalized Hilbert-Kunz multiplicity of the ideal $I$.

%\begin{Definition}
%    Suppose that $R$ is a ring of prime characteristic $p>0$. If $\{I_q\}$ is a $p$-family such that $I_{q'}\subset I_q$ for $q'>q$, then we will say that $\{I_q\}$ is a \textit{$p$-filtration}.
%\end{Definition}

\begin{remark}\label{lem6-5}
    Suppose that $R$ is a regular local ring of the prime characteristic $p>0$, and let $q = p^e$ for some $e\in\mathbb{N}$. Then for any ideals $I$ and $J$ in $R$, we have that
    \begin{equation*}
        (I\cap J)^{[q]}
        =(I^{[q]})\cap (J^{[q]}).
    \end{equation*}
\end{remark}
Now we extend Theorem \ref{thm6-1} to generalized Hilbert-Kunz multiplicities of $p$-families of ideals for regular local rings.

%want it for analytically unramified instead of regular
\begin{theorem}\label{thm6-2}
    Let $(R,m)$ be a $d$-dimensional regular local ring of prime characteristic $p>0$. Let $I_{\bullet} = \{I_q\}$ be a $p$-family in $R$ which satisfies condition $p(c)$ for some $c>0$. Suppose that for $P\in\ash(R)$ such that $P\supset I_1$, we have $P\supset I_q$ for all $q$. Then
    \begin{enumerate}
        \item[$(i)$]For all $q'$, the limit
        \begin{equation*}
            a_F(I_{q'}, (I_{q'})^{\sat}) = \lim_{q\to\infty}\frac{\ell_R([(I_{q'})^{\sat}]^{[q]}/(I_{q'}^{[q]}))}{q^d}
        \end{equation*}exists.
        \item[$(ii)$] The limit
        \begin{equation*}
            \lim_{q'\to\infty}\frac{a_F(I_{q'}, (I_{q'})^{\sat})}{(q')^d}
        \end{equation*}exists.
        \item[$(iii)$] We have the following formula
        \begin{equation*}
            e_{_gHK}(I_{\bullet})
            = \lim_{q'\to\infty}\frac{a_F(I_{q'}, (I_{q'})^{\sat})}{(q')^d}.
        \end{equation*}
    \end{enumerate}
\end{theorem}

\begin{proof}
    We begin by reducing to the case where $R$ is complete. For any ideal $J\subset R$, we may find $t>0$ sufficiently large so that $J^{\sat} = J:_Rm^t$ and\\ $(J\widehat{R})^{\sat} = J\widehat{R}:_{\widehat{R}}(m^t\widehat{R})$. By flatness of $R\to\widehat{R}$,
    \begin{equation*}
        J^{\sat}\widehat{R}
        = (J: m^t)\widehat{R}
        = J\widehat{R}: m^t\widehat{R}
        = (J\widehat{R})^{\sat}.
    \end{equation*} Consequently, we have that
    \begin{equation}
        \begin{split}
            \ell_R((I_{q'}^{\sat}))^{[q]}/I_{q'}^{[q]})
            = \ell_{\widehat{R}}
            ([(I_{q'}\widehat{R})^{\sat}]^{[q]}/(I_{q'}\widehat{R})^{[q]})
        \end{split}
    \end{equation}and
    \begin{equation}\label{eq9.3}
        \ell_R(I_q^{\sat}/I_q)
        = \ell_{\widehat{R}}((I_q\widehat{R})^{\sat}/I_q\widehat{R})
    \end{equation}for all $q$ and $q'$. At the begining of the proof of Theorem \ref{generalmult}, we showed that if $Q\in\ash(\widehat{R})$ and $I_1\widehat{R}\subset Q$, then $I_q\widehat{R}\subset Q$ for all $q$. Thus by replacing $\{I_q\}$ and $R$ with $\{I_q\widehat{R}\}$ and $\widehat{R}$, in Theorem \ref{thm6-2}, we may assume that $R$ is complete, and hence reduced.

    Let $J_{q} = I_q^{\sat}$, $J_{q',q} = (I_{q'}^{\sat})^{[q]}$, and $I_{q',q} = I_{q'}^{[q]}$ for all $q$ and $q'$. Let $\mu = \mu(m)$ be the minimal number of generators of $m$, and let $c_1 = \mu c$. Now we will show that
    \begin{equation}
        J_{q',q}\cap m^{c_1qq'} = I_{q',q}\cap m^{c_1qq'}
    \end{equation} for all $q$ and $q'$. Since $\{I_q\}$ satisfies $p(c)$, we have that
    \begin{equation*}
        (I_{q'}^{\sat}\cap m^{cq'})^{[q]}
         = (I_{q'}\cap m^{cq'})^{[q]}.
    \end{equation*}for all $q$ and $q'$. By Lemma \ref{lem6-5} we have that
    \begin{equation}\label{eq9.2}
        [(I_{q'})^{\sat}]^{[q]}\cap ( [m^{cq'}]^{[q]})
        = I_{q'}^{[q]} \cap ( [m^{cq'}]^{[q]}).
    \end{equation} On the other hand, we have
    \begin{equation*}
        m^{c_1qq'}=m^{\mu cqq'}\subset m^{[cqq']}
        = (m^{[cq']})^{[q]}
        \subset (m^{cq'})^{[q]}.
    \end{equation*}This combined with (\ref{eq9.2}) yields that
    \begin{equation*}
        [(I_{q'})^{\sat}]^{[q]}\cap m^{c_1qq'}
        = I_{q'}^{[q]} \cap m^{c_1qq'}.
    \end{equation*}As in
    \begin{equation*}
        J_{q',q}\cap m^{c_1qq'}
        = I_{q',q}\cap m^{c_1qq'}.
    \end{equation*}
    Now the conclusions of Theorem \ref{thm6-2} follow from Theorem \ref{generalmult}.
\end{proof}

%\begin{question}
%    Can the hypothesis of Theorem \ref{thm6-2} be relaxed to when $R$ is only analytically unramified (instead of regular), and when $I_1$ satisfies $p(c)$ (instead of $\{I_q\}$ satisfies $p(c)$)? The latter question has a positive answer if $I_q\subset (I^{[q]})^{\sat}$ for all $q$.
%\end{question}
%, that is, when the non $m$-primary components of $I_q$ are contained in those of $I^{[q]$
Now we discuss a graded version of the condition $p(c)$. In Remark \ref{rmk6.10} we will discuss the relationship between this condition and the $p(c)$ condition of ideals in local rings.

\begin{definition}\label{grpc}
    Let $A = \displaystyle \bigoplus_{i\geqslant 0} A_i$ be a graded ring of prime characteristic $p>0$, $J\subset A$ an ideal. Let $A_+ = \displaystyle \bigoplus_{i> 0} A_i$. Define the saturation of $J$ to be
    \begin{equation*}
        J^{\sat}:= J:A_+^{\infty} = \cup_{i>0} J:A_+^{i}.
    \end{equation*}
    If $c$ is a positive integer, then we will say that $J$ satisfies $p(c)$ if $(J^{[q]})^{\sat}\cap A_+^{qc}=J^{[q]}\cap A_+^{qc}$ for all $q=p^e\geqslant 1$.
\end{definition} 

%this can be generalized a fair bit
\begin{remark}\label{rmk6.10}
    Let $k$ be a field of characteristic $p>0$, take $A = k[x_1,\ldots,x_n]$ and let $J\subset A$ be an ideal. Let $m = A_+$ ($=(x_1,\ldots,x_n)$) and let $I = J_m$. If $c>0$ and $J$ satisfies $p(c)$, then $I$ satisfies $p(c)$ as an ideal in the local ring $R := A_m$.
\end{remark}
\begin{proof}
    Let $m_R$ be the maximal ideal of $R$, so that $m_R = m_m$. We have $(J^{[q]})_{m} = (J_m)^{[q]}$, so that we may write $J^{[q]}_m$ without risk of confusion. Fix $q = p^e\geqslant 1$. By assumption, we have $(J^{[q]})^{\sat}\cap m^{cq}\subset J^{[q]}$ so that
    \begin{equation}\label{eqex1}
        ((J^{[q]})^{\sat}\cap m^{cq})_{m}\subset J^{[q]}_{m}
    \end{equation} We have $((J^{[q]})^{\sat}\cap m^{cq})_{m} = ((J^{[q]})^{\sat})_{m}\cap m^{cq}_m = ((J^{[q]})^{\sat})_{m} \cap m_R^{cq}$. Next we will show that $((J^{[q]})^{\sat})_m = (J^{[q]}_m)^{\sat}$. Since $R$ is Noetherian, we have for $t>>0$ that
    \begin{equation*}
        (J^{[q]})^{\sat} = J^{[q]}:m^t\text{ and } (J^{[q]}_m)^{\sat} = J^{[q]}_m : (m^t_m).
    \end{equation*} Thus,
    \begin{equation*}
        \begin{split}
            ((J^{[q]})^{\sat})_m
            &= \{y/s\mid y\in A\text{, }ym^t\subset J^{[q]}\text{, }s\in A\setminus m\}
            \\&= \{y/s\mid y\in A,\text{, }s\in A\setminus m\text{, }(y/s)m^t_m\subset J^{[q]}_m\}
            = (J^{[q]}_m)^{\sat}.
        \end{split}
    \end{equation*}Combining this with (\ref{eqex1}) yields that
    \begin{equation*}
        (I^{[q]})^{\sat}\cap m_R^{cq}\subset I^{[q]}
    \end{equation*}and hence
    \begin{equation*}
        (I^{[q]})^{\sat}\cap m_R^{cq} = I^{[q]}\cap m_R^{cq}
    \end{equation*}so that $I$ satisfies $p(c)$.
\end{proof}

%\begin{question}
%    Does the converse of Remark \ref{rmk6.10} hold? That is, if $c>0$ and $I$ satisfies $p(c)$, then does $J$ satisfy $p(c)$?
%\end{question}

\begin{example}\label{ex6.11}
    Let $k$ be a field of characteristic $p>0$, and let $A = k[x,y,z]$. Let $J = (x^p,xyz,y^p)$. We will show that $J$ satisfies $p(c)$, where $c:= 4p$. We have $J^{[q]} = (x^{pq},x^qy^qz^q,y^{pq})$. Consequently, we see that
    \begin{equation*}
        J^{[q]} = (x^{pq},x^qy^q,y^{pq})
        \cap (x^{pq},z^q,y^{pq})
    \end{equation*}which is an irredundant primary decomposition (since the radicals of the ideals on the right hand side are $(x,y)$ and $(x,y,z) = m_R$). Thus,
    \begin{equation*}
        (J^{[q]})^{\sat}
         = (x^{pq},x^qy^q,y^{pq}).
    \end{equation*} Now we must show that $(x^{pq},x^qy^q,y^{pq})\cap m^{cq}\subset J^{[q]}$, which will show that $J$ satisfies $p(c)$. Let $\alpha \in (x^{pq},x^qy^q,y^{pq})\cap m^{4pq}$ and write $\alpha = x^{pq}f+x^qy^qg+y^{pq}h$, where $f,g,h\in A$. Since $x^{pq},y^{pq}\in J^{[q]}, $we have $\alpha \in J^{[q]}$ iff $x^qy^qg\in J^{[q]}$. Write $g = g_1+\cdots+g_r$ where $g_1,\ldots,g_r$ are monomials. Fix $i\in\{1,\ldots,r\}$. Since $m^{4pq}$ is a monomial ideal, we have $x^qy^qg_i\in m^{4pq}$ so that $\deg(g_i)\geqslant 4pq-2q\geqslant 4pq-pq\geqslant 3pq$. Thus $g_i\in m^{3pq}$, so that one of $x^{pq},y^{pq},z^{pq}$ divides $g_i$. Consequently, $x^qy^qg_i\in (x^{pq},y{pq},x^qy^qz^{pq})\subset J^{[q]}$. Since this holds for all $i$, we now have $x^qy^qz^q\in J^{[q]}$ so that $\alpha \in J^{[q]}$. Thus, $(J^{[q]})^{\sat}\cap m^{4pq}\subset J^{[q]}$, so that $(J^{[q]})^{\sat}\cap m^{cq} = J^{[q]}\cap m^{cq}$. Therefore $J$ satisfies $p(c)$. Then by Remark \ref{rmk6.10}, $I: = J_{(x,y,z)}$ satisfies $p(c)$ so that the conclusions of Theorem \ref{thm6-1} hold for $I$.
 \end{example}
 %, so that the conclusions of Theorem \ref{thm6-1} hold for $I$.

\begin{example}\label{ex6.12}
Let $A$, $m$, and $R$ be as in Example \ref{ex6.11}, and suppose that $I = J\cap \mathfrak{a}$, where $\mathfrak{a}$ is $m_R$-primary, and $J$ is $(x,y)$-primary. Then $I_m$ satisfies $p(l\mu)$ where $l$ is any positive integer such that $\mathfrak{a}\supset m^l$ and $\mu:= \mu(\mathfrak{a})$.
\end{example}

\begin{proof}
    Let $l$ be a positive integer such that $\mathfrak{a}\supset m^l$. Since $A$ is a regular ring, the map $A\to A^{1/q}$ is faithfully flat for all $q$. Then the proof of Lemma \ref{lem6-5} shows that
    \begin{equation}\label{eq 41-1}
        I^{[q]} = \mathfrak{a}^{[q]}\cap J^{[q]}.
    \end{equation}for all $q$. By hypothesis, $J^{[q]}$ is $(x,y)$-primary and $\mathfrak{a}^{[q]}$ is $m$-primary. Thus the equation (\ref{eq 41-1}) is an irredundant primary decomposition of $I^{[q]}$. Thus, $(I^{[q]})^{\sat} = J^{[q]}$. Since $\mathfrak{a}\supset m^l$, we have $\mathfrak{a}^{[q]}\supset \mathfrak{a}^{\mu q}\supset m^{\mu l q}$. Let $c = l\mu$. Then
    \begin{equation*}
        (I^{[q]})^{\sat}\cap m^{cq}
        \subset J^{[q]}\cap \mathfrak{a}^{[q]}
        = I^{[q]}
    \end{equation*}so that
    \begin{equation*}
        (I^{[q]})^{\sat}\cap m^{cq}
        = I^{[q]}\cap m^{cq}.
    \end{equation*}Therefore $I$ satisfies $p(c)$ in the graded ring $A$. Then by Remark \ref{rmk6.10}, $I_m$ satisfies $p(c)$ in the ring $A_m = R$.
\end{proof}

\section{Application to BBL weakly $p$-families}\label{sec8}

In this section, we explore some additional applications of Theorem \ref{generalmult} in Theorem \ref{thm7-1} for bounded below linearly weakly $p$-families. For more information on these types of families, we refer  \cite{DC,Tucker}.

\begin{definition}(Definition 4.5 \cite{DC})
    Let $R$ be a local ring and let $I_{\bullet}=\{I_n\}$ be a family of ideals. We say that $I_{\bullet}$ is \textit{bounded below linearly}, or BBL for short, if either (i) $I_{\bullet}$ is indexed by the natural numbers and there exists $c>0$ such that $m^{cn}\subset I_n$ for all $n$, or (ii) $\text{char}(R)$ is a prime $p>0$, $I_{\bullet}$ is indexed by the natural powers of $p$, and there exits $c>0$ such that $m^{cq}\subset I_q$ for all natural powers $q$ of $p$.
\end{definition}

\begin{definition}\label{df-weakp} (Definition 4.4 \cite{DC})
    Let $R$ be a Noetherian local ring of characteristic $p>0$ and let $I_{\bullet}=\{I_q\}$ be a family of ideals indexed by the natural powers of $p$. We say that $I_{\bullet}$ is a \textit{weakly $p$-family} of ideals if there exists $c\in R^\circ$ such that for any $q = p^e$ we have that $cI_q^{[p]}\subset I_{pq}$.
\end{definition}
\begin{example} \label{weakly_p_example}
For two ideals $I,J \subset R$,  if $J \cap R^\circ \neq \emptyset$, then $I^{[q]}:J$ is a weakly $p$-family of ideals.   
\end{example}
\begin{remark}
We see that if $I_{\bullet}$ is a weakly $p$-family, and $c\in R^{\circ}$ is as in Definition \ref{df-weakp} and $J_q:= cI_q$ for all $q$, then for all $q=p^e$ with $e\geqslant 1$ we have 
\begin{equation*}
    J_q^{[p]} = (cI_q)^{[p]}=c^pI_q^{[p]}\subset c^2I_q^{[p]}\subset cI_{pq} = J_{pq}.
\end{equation*} Hence $\{J_q = cI_q\}$ is a $p$-family.
\end{remark}
\begin{theorem}\label{thm7-1}
    Let $(R,\mathfrak{m})$ be a $d$-dimensional Noetherian local ring of characteristic $p>0$. Let $I_{\bullet} = \{I_q\}$ be a BBL weakly $p$-family of $\mathfrak{m}$-primary ideals in $R$. Suppose that the dimension of the nilradical of $\widehat{R}$ is less than $d$. Then we have the following.
    
    \begin{enumerate}%fix italics in the numbers
    \item[$(1)$]
        $\lim_{q\to\infty}\frac{\ell_R(R/I_q)}{q^d}$
    exists.
    
    \item[$(2)$]
    $\lim_{q\to\infty}\frac{\ell_R(R/I_q)}{q^d} = \lim_{q\to\infty}\frac{e_{HK}(I_q)}{q^d}$.
    \end{enumerate}
\end{theorem}

\begin{proof}
We first reduce to the case where $R$ is a complete local domain. There are two important considerations to address. Firstly, we must ensure that the operations performed do not alter the property of the family, i.e., it remains a BBL weakly $p$-family. Secondly, we must verify that the asymptotic colength remains unchanged. These points can be observed from the same steps described in Lemmas 10.1, 10.2, and 10.3 of \cite{DC}. 
Since $\{I_q\}$ is a weakly $p$-family, we can choose a $c \neq 0$ such that $c(I_q^{[p]})\subset I_{pq}$ for all $q$. Note that $ \ell (R/I_q)= \ell (cR/cI_q)$. We define $I_{\bullet}^{'}= \{cI_{q}\}$ and $J_{\bullet}^{'}=\{cR\}$ for all $q$. It is easy to see that $I_{\bullet}^{'}$ and $J_{\bullet}^{'}$ are $p$-families. We claim that there exist a positive $k$ such that $ cI_{q} \cap \mathfrak{m}^{kq} = cR \cap \mathfrak{m}^{kq}$ for all $q$. One direction is clear. Now for the other part, using the Artin-Rees lemma for large enough $k$, we have $cR \cap \mathfrak{m}^{kq} = \mathfrak{m}^{c_0}(cR \cap \mathfrak{m}^{kq-c_0}) \subset c\mathfrak{m}^{c_0} \subset cI_{q}$ for all $q$, and $c_0$ is the constant which satisfies the property of BBL family for $\{I_q\}$. Now using \cite[Theorem 5.12]{HJ} we get $(1)$.

To prove (2), we will employ Theorem \ref{generalmult}. We define $I_{\bullet}^{'}$ and $J_{\bullet}^{'}$ as before. Now, let $I_{q’, \bullet}^{'} = \{c^qI_{q’}^{[q]}\}$ and $J_{q’, \bullet}^{'} = \{c^qR\}$. It is evident that $J_{q’,1}^{'} = J_{q’}^{'}$, $J_{q’,q}^{'}\subset J_{qq’}^{'}$, $I_{q’,1}^{'} = I_{q’}^{'}$, and $I_{q’,q}^{'}\subset I_{qq’}^{'}$ for all $q$. We will now demonstrate that there exists a positive number $k_0$ such that
\begin{equation*}
     c^qI_{q’}^{[q]} \cap  \mathfrak{m}^{k_0qq'} = c^qR \cap  \mathfrak{m}^{k_0qq'}
\end{equation*}for all $q,q’$. One containment is always true, and for the other, we set $k_0 = c_0 + r + k + 1$, where $c_0$ and $k$ are the constants used in the proof of $(1)$, and $r$ is the minimum number of generators of the maximal ideal $\mathfrak{m}$. Now,  $c^qR \cap \mathfrak{m}^{k_0qq'} = \mathfrak{m}^{k_0qq^{'}-k}(c^qR \cap \mathfrak{m}^{k}) \subset c^q \mathfrak{m}^{k_0qq^{'}-k}  $ and note that $k_0qq^{'}-k \geqslant (c_0q^{'})(r+1)q $. This implies  $c^qR \cap \mathfrak{m}^{k_0qq^{'}} \subset c^q(\mathfrak{m}^{c_0q^{'}})^{(r+1)q} \subset c^q(\mathfrak{m}^{c_0q^{'}})^{[q]} \subset c^qI_{q^{'}}^{[q]}$. Second containment in the above chain follows from pigeon-hole principle. Let's say $\{x_1,...,x_r\}$ are the generators of $\mathfrak{m}$. Now any generator of 
$(\mathfrak{m}^{c_0q^{'}})^{(r+1)q}$ will be of the form $x_1^{k_1}...x_r^{k_r}$ such that $k_1+...+k_r = c_0(r+1)qq^{'}$. This implies that there exists at least one $i$ such that $k_i > c_0qq^{'}$, and $x_i^{c_0qq^{'}}$ is one of the generators of $(\mathfrak{m}^{c_0q^{'}})^{[q]}$. Therefore, we get the second containment and the desired equality. Since we operate within a complete local domain, and all ideals are $\mathfrak{m}$-primary, the last condition of Theorem \ref{generalmult} is automatically fulfilled.  
Following the notation used in Theorem \ref{generalmult} we define $\cal{F}(I_{\bullet}^{'},J_{\bullet}^{'}):= \lim_{q\to\infty}\frac{\ell_R(J_q^{'}/I_q^{'})}{q^d} = \ell (cR/cI_q) $ and $ \cal{G}(J_{q',\bullet}^{'}, I_{q',\bullet}^{'}):= \displaystyle \lim_{q\to\infty}
    \dfrac{\ell_R(J_{q',q}/I_{q',q})}{q^d} = \displaystyle \lim_{q \to \infty} \dfrac{ \ell \left( c^qR/c^qI_{q^{'}}^{[q]} \right)}{q^d} = \displaystyle \lim_{q \to \infty} \dfrac{\ell \left( R/I_{q^{'}}^{[q]}\right)}{q^d} = e_{HK}(I_{q^{'}})$. So $(2)$ follows using the conclusion of Theorem \ref{generalmult}.
\end{proof}%is it c(I_q^{[p})) or (cI_q)^{[p]} \subset I_pq ?

%%%%%%%%%%%%%%%%%%%%%%%%%%%%%%%%%%%%%%%%%
\subsection*{Acknowledgments}
%%%%%%%%%%%%%%%%%%%%%%%%%%%%%%%%%%%%%%%%%

The second author thanks the Center for Mathematical Sciences and Applications for their support during the writing of this manuscript.

\end{document}